\documentclass[a4paper,12pt]{article}

\usepackage[T1]{fontenc}
\usepackage[english]{babel}
\usepackage{lmodern}

\usepackage{amssymb, amsmath, amsthm, bbm}
\usepackage{mathtools}
\usepackage{mathrsfs, dsfont}

\usepackage[top=1.8cm, bottom=2.8cm]{geometry}

\usepackage{tabularx}     
\usepackage{enumerate}
\usepackage{graphicx}
\usepackage{ifpdf}

\usepackage[margin=1cm]{caption} 

\usepackage{xcolor}
\usepackage[colorlinks,
  linkcolor=blue!60!black,
  citecolor=red!65!black,
  urlcolor=green!50!black]{hyperref}
\usepackage{cleveref}

\usepackage[numbers,sort&compress]{natbib}
\usepackage{doi}

\newtheorem{theorem}{Theorem}
\newtheorem{corollary}[theorem]{Corollary}
\newtheorem{lemma}[theorem]{Lemma}
\newtheorem{proposition}[theorem]{Proposition}

\newtheorem{definition}[theorem]{Definition}
\newtheorem{remark}[theorem]{Remark}

\newtheorem*{remark*}{Remark}

\newcommand{\NN}{\mathbb{N}}
\newcommand{\RR}{\mathbb{R}}
\renewcommand{\P}{\mathbb{P}}
\newcommand{\EE}{\mathbb{E}}

\newcommand{\diff}{\mathrm{d}}
\newcommand{\indic}[1]{\mathbbm{1}_{\{#1\}}}

\DeclarePairedDelimiter\abs{\lvert}{\rvert}

\renewcommand{\d}{\mathrm{d}}

\DeclareMathOperator*{\argmin}{arg\,min}
\DeclareMathOperator{\Card}{Card}

\newcommand{\Out}{\mathrm{out}}
\newcommand{\In}{\mathrm{in}}

\title{General Epidemiological Models: Law of Large Numbers and Contact Tracing}

\usepackage{authblk}

\newlength{\affilskip}
\author[1]{Jean-Jil Duchamps}
\author[2]{Félix Foutel-Rodier}
\author[3]{Emmanuel Schertzer}

\affil[1]{
    Université de Franche-Comté, CNRS, LmB, F-25000 Besançon, France
    \vspace{\affilskip}
}

\affil[2]{
    Department of Statistics, University of Oxford, UK
    \vspace{\affilskip}
}
\affil[3]{
    Faculty of Mathematics, University of Vienna, \authorcr
    Oskar-Morgenstern-Platz 1, 1090 Wien, Austria
    \vspace{\affilskip}
}

\date{}

\begin{document}

\maketitle

\begin{abstract}
We study a class of individual-based, fixed-population size epidemic models under general assumptions, e.g., heterogeneous contact rates encapsulating changes in behavior and/or enforcement of control measures. We show that the large-population dynamics are deterministic and relate to the Kermack--McKendrick PDE. Our assumptions are minimalistic in the sense that the only important requirement is that the basic reproduction number of the epidemic $R_0$ be finite, and allow us to tackle both Markovian and non-Markovian dynamics. The novelty of our approach is to study the ``infection graph'' of the population. We show local convergence of this random graph to a Poisson (Galton--Watson) marked tree, recovering Markovian backward-in-time dynamics in the limit as we trace back the transmission chain leading to a focal infection. This effectively models the process of contact tracing in a large population. It is expressed in terms of the Doob $h$-transform of a certain renewal process encoding the time of infection along the chain. Our results provide a mathematical formulation relating a fundamental epidemiological quantity, the generation time distribution, to the successive time of infections along this transmission chain.
\end{abstract}

\tableofcontents

\section{Introduction}

\subsection{General individual-based epidemic model}

In the present article, we study an extension of the general epidemiological framework
introduced in \cite{foutelrodier2020individualbased} to model the
COVID-19 epidemic. Let us briefly recall the main features of this model.

At time $t=0$, we consider a population made of susceptible individuals, that have never
encountered the disease, and infected individuals. Each infected
individual is supposed to belong to one \emph{compartment}, that models
the stage of the disease of this individual. Classical examples of compartments
are the exposed compartment ($E$), where the individual is
infected but not yet infectious, the infectious compartment ($I$), and the
recovered compartment ($R$), once the individual has become immunized. In the case of the COVID-19 epidemic, it might be relevant to
add a hospitalized ($H)$ and an intensive care unit ($U$) compartment, as
monitoring the number of individuals in these states is typically important for
policy-making. See \Cref{fig:model_covid} for an example of compartmental model
used for the COVID-19 epidemic. We denote by $\mathcal{S}$ the set of all
compartments. For the sake of simplicity, we will also assume that $\mathcal{S}$ is finite.

We encode the compartment to which individual $x$ belongs as a stochastic
process $(X_x(a);\, a \ge 0)$ valued in $\mathcal{S}$, that we call the
\emph{life-cycle process}. The random variable $X_x(a)$ gives the compartment
to which $x$ belongs at \emph{age of infection} $a$, that is, $a$ unit of time
after its infection. Moreover, individual $x$ is endowed with a point measure
$\mathcal{P}_x$ on $\RR_+$ that we call the \emph{infection point
process}, which is assumed to be simple.
The atoms of $\mathcal{P}_x$ encode the age at which $x$ makes infectious contacts
with the rest of the population.
We think of the pair $(\mathcal{P}_x, X_x)$ as describing the course of the
infection of individual $x$. We make the fundamental assumption that the pairs
$(\mathcal{P}_x, X_x)$ are i.i.d.\ for distinct individuals in the population.

\begin{figure}[t]
    \centering
    \includegraphics[width=.95\textwidth]{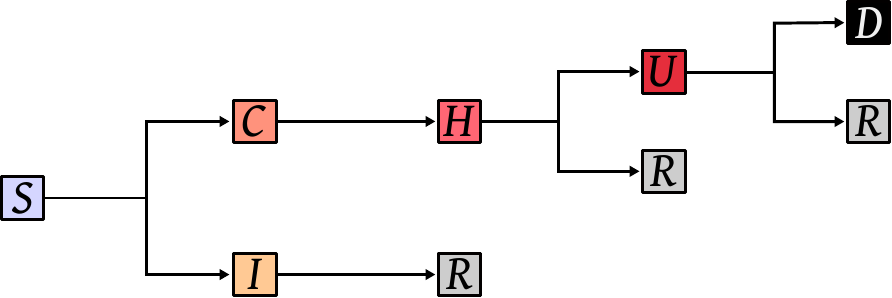}
    \caption{An example of compartmental model for the COVID-19 epidemic.
        The compartments are: $S$, susceptible; $I$, mildly infected;
        $C$, severely infected; $H$, hospitalized; $U$, admitted to
        intensive care unit; $R$, recovered; and $D$, dead.}
    \label{fig:model_covid}
\end{figure}

In \cite{foutelrodier2020individualbased} we assumed that susceptibles are
in excess, and that any infectious contact leads to a new infection.
The resulting population is then distributed as a Crump-Mode-Jagers (CMJ)
process. In the current work, we consider an extension of this model that takes
into account the saturation due to the finite pool of susceptibles in the
population. More precisely, we consider a population of finite fixed size $N$.
Each infectious contact is made with an individual uniformly chosen in this
population, and it results in a new infection only if the targeted individual
is susceptible. Finally, we model the impact of control measures, such as
school closure, or national lockdown, with a \emph{contact rate function}
$(c(t);\, t \ge 0)$. This contact rate is such that an infection
occuring at time $t$ is only effective with probability $c(t) \in [0, 1]$.
With probability $1-c(t)$, the infection is removed. A formal description of
this model is provided in \Cref{SS:descriptionModel}.

\subsection{Law of large numbers for the age structure} 

A standard way to study compartmental models is to consider the dynamics of
the number of individuals in each compartment. If the underlying
probabilistic model is Markovian, this typically gives rise to systems of
ODEs of the SIR type in the large population size limit, see
\cite{britton2019} for a recent account. Here, we will not keep track
of the count of individuals in the various compartments, but we will rather be
interested in the age structure of the population. Our main result is a law
of large number for the age structure of population, which is the equivalent of
Theorem~7 of \cite{foutelrodier2020individualbased} for our non-linear
extension of the model.

We anticipate the notation of \Cref{SS:descriptionModel} and denote the
empirical measure of ages and compartments in the population at time $t$
as
\[
    \forall i \in \mathcal{S},\quad \mu^N_t(\d a, \{i\}) 
    = \sum_{x=1}^N \indic{\sigma^N_x \le t} \indic{X_x(t-\sigma^N_x)=i}\delta_{t-\sigma^N_x}(\d a),
\]
where $\sigma^N_x$ is the infection time of individual $x$, and the sum runs over all
infected individuals at time $t$. (Note that $t-\sigma^N_x$ is just the age of 
$x$ at time $t$.) The measure $\mu^N_t$ encodes the ages and compartments of
infected individuals at time $t$. Let us also denote by 
\[
    \forall t \ge 0,\, \forall i \in \mathcal{S},
    \quad Y^N_t(i) = \# \{\text{individuals in $i$ at time $t$}\}
    = \mu^N_t\big( [0, \infty), \{ i \} \big).
\]
The limiting distribution of $\mu^N_t$ will depend on the following two
quantities:

\begin{itemize}
    \item The intensity measure of the infection point process defined as 
        \begin{equation} \label{eq:intensity}
            \tau(\d a) \coloneqq \EE \bigl[ \mathcal{P}(\d a) \bigr].
        \end{equation}
        We assume that $\tau$ has a density w.r.t.\ the Lebesgue
        measure that we denote by $\tau(a)$ with a slight abuse of notation,
        and that $R_0 \coloneqq \tau([0, \infty)) < \infty$. We
        also assume that there exists $\alpha \in \RR$, called the
        \emph{Malthusian parameter} such that 
        \begin{equation} \label{eq:malthusian}
            \int_0^\infty e^{-\alpha a} \tau(a) \diff a = 1.
        \end{equation}

    \item The one-dimensional marginals of the life-cycle process,
        denoted by 
        \[
            \forall i \in \mathcal{S},\, \forall a \ge 0,\quad 
            p(a, i) \coloneqq  \P\bigl( X(a) = i \bigr).
        \]
\end{itemize}

Let $I_0\in(0,1)$. At time $t=0$, we assume that every individual is
independently infected with probability $I_0$ and that its age of
infection is chosen independently according to a probability density $g$
on $\RR_+$. In the following, we define $\mathcal{I}^N_0 \subseteq [N]$ as
the set of infected individuals at $t=0$. We will also use the notation
$S_0 = 1-I_0$.

We make the simplifying assumption that the underlying compartmental
model is acyclic: we assume that for any two compartments $i, j \in
\mathcal{S}$, if $j$ can be accessed from $i$ with positive probability,
that is, if the event that we can find $s \le t$ such that $X(s) = i$ and
$X(t) = j$ has positive probability, then $i$ cannot be accessed from
$j$. In other words, the directed graph on $\mathcal{S}$ composed of all
edges $i\to j$ such that $j$ is accessible from $i$ is a directed acyclic
graph. This assumption is not very restrictive, most natural
compartmental models enjoy this acyclic property. See
Figure~\ref{fig:model_covid}. It is only needed to reinforce the
finite-dimensional convergence to a Skorohod one in the following theorem. 

\begin{theorem} \label{thm:ageStructure}
    As $N \to \infty$, the following convergence holds in
    probability in the Skorohod topology on the space of measure-valued
    processes,
    \[
        \Big(\frac{1}{N} \mu^N_t(\d a, \{i\});\, t \ge 0\Big) 
        \longrightarrow 
        \big( n(t,a)p(a,i) \,\d a;\, t \ge 0 \big)
    \]
    
    where $(n(t,a);\, t, a \ge 0)$ is the solution to
    \begin{align} \label{eq:mckvf}
        \begin{split}
        \partial_t n(t,a) + \partial_a n(t,a) &= 0 \\
        \forall t \ge 0,\; n(t, 0) &= c(t)S(t) \int_0^\infty n(t,a) \,\tau(\d
        a) \\
        \forall a \ge 0,\; n(0, a) &= I_0 g(a)\\
        \forall t \ge 0,\; S(t) &= 1 - \int_0^\infty n(t,a) \,\d a.
        \end{split}
    \end{align}
\end{theorem}

This theorem is proved at the end of \Cref{SS:proofMainThm}, as a
consequence of the convergence of the so-called historical process
(\Cref{thm:hist-process}). From \Cref{thm:ageStructure} we immediately
deduce the following result.

\begin{corollary}
For any $i \in \mathcal{S}$, we have
    \begin{equation} \label{eq:limitCompartments}
        \Big(\frac{1}{N}Y^N_t(i);\, t \ge 0\Big) 
        \longrightarrow 
        \Big(\int_0^\infty n(t,a) p(a,i) \,\d a;\, t \ge 0\Big)
    \end{equation}
    in probability in the Skorohod topology.
\end{corollary}

\begin{remark}
\Cref{thm:ageStructure} can be easily extended to weaker assumptions on
the initial condition. For instance, it is not hard to see from our proof
that \Cref{thm:ageStructure} holds true if we simply assume that
\[
    \frac{1}{N} \sum_{x =1}^N \indic{\sigma^N_x < 0}
    \delta_{-\sigma_x^N}(\diff a) \longrightarrow I_0 g(a) \diff a
\]
where the convergence holds in probability for the weak topology. 
\end{remark}

\begin{remark*}[Infinite reproduction number]
    The dynamics of the epidemic until some fixed time $t$ does not
    depend on the potential infections occurring after time $t$. In
    particular we can remove all atoms after age $t$ of the point
    processes of the individuals that are initially susceptible without
    affecting the process at time $t$. Similarly, we can remove all atoms
    after time $t$ of the point processes of the initially infected
    individuals. Anticipating the notation $\bar{\tau}$ in
    \eqref{eq:tauBar} for the intensity measure of the initially infected
    individuals, the convergence of the process in
    Theorem~\ref{thm:ageStructure} until time $t$ only requires that
    $\int_0^t \tau(a) \diff a < \infty$ and $\int_0^t \bar{\tau}(a) \diff
    a < \infty$. Thus, Theorem~\ref{thm:ageStructure} holds under the
    weaker assumption that $\tau$ is a locally finite measure, and that
    $g$ decays fast enough so that $\bar{\tau}$ is also locally finite.
    In particular one could have that $R_0 = \infty$.
\end{remark*}

This result will follow from the more general \Cref{thm:hist-process},
and is proved in \Cref{SS:proofMainThm}. The definition of a solution to
\Cref{eq:mckvf} is provided in \Cref{SS:mckvf}.
Theorem~\ref{thm:ageStructure} proves that the age structure of the
population converges to a limiting non-linear PDE of the
Kermack--McKendrick type \cite{kermack1927}. It also entails that the
number of individuals in each compartment can be recovered by integrating
the one-dimensional marginals $p(a,i)$ against the age structure. 

It is interesting to note that the limit of our model is universal.
The limiting expression in Equation~\eqref{eq:limitCompartments} does not
depend on the entire distribution of the pair $(\mathcal{P}, 
(X(a);\, a \ge 0))$, but only on:
\begin{itemize}
    \item the \emph{mean} number of secondary infections $\tau(a)$
        induced by an infected individual with age $a$;
    \item the one-dimensional marginals $p(a,i)$ of the life-cycle
        process.
\end{itemize}
These are the only individual characteristics that need to be assessed to
forecast the dynamics of the epidemic at a large scale. By further writing
$\tau = R_0 \nu$ with 
\begin{equation} \label{eq:R0}
    R_0 = \int_0^\infty \tau(\d a),\quad \nu(\d a) =
    \frac{\tau(\d a)}{R_0},
\end{equation}
we see that $\tau$ only depends on two well-known epidemiological
quantities:
\begin{itemize}
    \item the basic reproduction number, $R_0$, which the mean number
        of secondary infections induced by a single individual in an
        entirely susceptible population;
    \item the distribution of the generation time, $\nu$, which gives the
        typical time between the infection of a source individual and
        that of the recipient individual in an infection pair
        \cite{fraser2007}.
\end{itemize}
Further interesting modeling consequences of Theorem~\ref{thm:ageStructure} are
discussed in our earlier work \cite{foutelrodier2020individualbased}.

\subsection{Contact tracing: the historical process}
\label{SS:contactTracing}

We already argued that our approach allows to identify the individual
characteristics that impact the large population size dynamics. We
identified those parameters as $R_0$, the distribution of generation time
$\nu$ together with the one-dimensional marginals of the life-cycle
process. The estimation of those parameters is obviously of paramount
importance. One possible approach to estimate the generation time
distribution consists in observing the generation times backwards in time
using contact tracing \cite{ferretti_quantifying_2020,ganyani2020estimating}, 
i.e., the time between the infection time of an individual (the infectee)
and that of his/her infector (rather than the infection time of the
individuals he/she infects). In \cite{britton}, the authors addressed
this specific question in a simplified setting. More specifically, they
assumed that $c \equiv 1$ and that the susceptibles are in excess so that
our microscopic model can be approximated by a Crump--Mode--Jagers
process as in our earlier work \cite{foutelrodier2020individualbased}. 
They showed that the observation of backward generation times raises two
serious issues:
\begin{itemize} 
\item[(i)] First, observations of past infections induce a strong
    observational bias: the backward generation time distribution differs
    from the actual generation time distribution. In the supercritical
    case (i.e., when $R_0>1$), the backward generation time has density
    \begin{equation}\label{eq:backward-g} 
        \exp(-\alpha u) R_0 \nu(u)
    \end{equation} where $\alpha>0$ is the Malthusian parameter of the
    model defined in \eqref{eq:malthusian}. As a consequence,
    observations of backward infection times tend to be biased towards
    lower values.  
\item[(ii)] Infection times are difficult to observe. Instead, the onset
    of symptoms is generally observed. For this reason, the serial
    interval, which is defined as the time between symptom onsets in the
    two individuals mentioned above, is often used as a surrogate for the
    generation time. As discussed in \cite{britton}, this can induce a
    second source of significant observational bias. 
\end{itemize}
As already mentioned above, the authors in \cite{britton} address the
previous bias in the case where $c \equiv 1$ and when susceptibles are in
excess. In the present article, we will show if we (1) take into account
saturation effect (i.e, when the population is out of the branching
process regime), and (2) assume some heterogeneity in the contact rate,
then those two components of the dynamics can induce a third source of
bias. 

In order to provide some intuition of the upcoming results, consider a
newly infected individual at time $t$. Trace backward in time the chain
of infection up to time $t=0$. (The first individual along the chain is
the infector of the focal individual, the second is the infector of the
primary infector and so on.) Finally, along the chain, report the
successive times of infection, see \Cref{fig:ancestors}. When
susceptibles are in excess (branching approximation) Jagers and
Nerman~\cite{Nerman1984} showed under mild assumptions that as
$t\to\infty$, the successive time of infections are well approximated by
the values of a renewal process
\[
\mathcal{R}^{(t)}(0)=t, \ \ \mathcal{R}^{(t)}(k) = t- \sum_{i=1}^k \xi_i \mbox{  for $k\geq1$},
\]
where the $\xi_i$'s are i.i.d.\ and distributed according to \eqref{eq:backward-g}.

\begin{figure}[ht] \centering
  \includegraphics{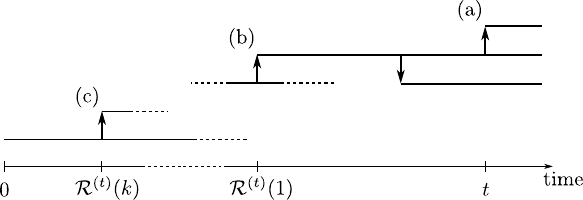}
  \caption{Chain of infection of a focal individual. Horizontal lines represent the lifetime of an infected individual, and vertical arrows represent new infections. Going backwards in time, the following events are recorded: (a) the focal individual $i_0$ is infected at time $t$; (b) its infector, individual $i_1$, is infected at time $\mathcal{R}^{(t)}(1)$; (c) going back a chain of $k$ successive infections, individual $i_k$ is infected at time $\mathcal{R}^{(t)}(k)$ by an individual that was already infected at time $0$.}
  \label{fig:ancestors}
\end{figure}

In the presence of saturation, we show that the chain of infection is
given by an $h$-transform of the renewal process $\mathcal{R}^{(t)}$.
Intuitively, the $h$-transform tends to favor infection at times where
there is a large fraction of susceptibles and  a high contact rate. When
the initial age structure of the population coincides with the
``equilibrium'' measure of the branching approximation, i.e.,
\[
    g(u) = \alpha \exp(-\alpha u),
\]
the $h$-transformed process can be reformulated in a simple manner. In
\Cref{prop:exx} we show that it is identical in law to the original
renewal process conditioned on survival assuming that at each step $k$
the process is killed with probability 
$1-c(\mathcal{R}^{(t)}(k)) S(\mathcal{R}^{(t)}(k))$.

In \Cref{SS:proofMainThm}, we introduce the historical process. Loosely
speaking, the historical process is the empirical measure reporting the
chain of infections for every individual infected at time $t=0$ or who
eventually gets infected in the future. It is constructed by reporting
the successive age of infections along the chain, but also the stages of
the life-cycle process for every ``ancestor'' along the chain, e.g. onset
of the symptoms, latency period, etc. In \Cref{thm:hist-process}, we show
that the historical process converges to a deterministic probability law.
Loosely speaking, our result shows that the chain of infections for a
finite sample of infected individuals are asymptotically independent.
Furthermore, for each sampled individual, its chain of infection is
distributed in such a way that
\begin{itemize} 
    \item the successive times of infection are expressed in terms of the
        $h$-transformed renewal process mentioned above.
    \item the life cycle of individuals along the chain is biased, and
        the bias can be expressed as a Palm modification of the original
        life cycle. This will be made more formal in \Cref{SS:limitTree}.
\end{itemize} 
Going back to the two epidemiological questions (i) and (ii), our results
decipher how the epidemiological parameters $R_0$ and the generation time
distribution $\nu$ relate (in a non-trivial way) to observables which can
be directly collected from contact-tracing data.

\subsection{A genealogical dual to the delay equation}
\label{SS:dual}

The Kermack--McKendrick equation~\eqref{eq:mckvf} can be reformulated
in terms of a non-linear delay equation. 
To ease the exposition, let us consider the case $c\equiv 1$.
In \Cref{SS:limitTree}, the general case $c \not\equiv 1$ will be exposed. 

If $(n(t,a);\, t,a\ge0)$ denotes
the solution to \Cref{eq:mckvf} with $c \equiv 1$,  let us
define the number of infections between time $0$ and $t$ as 
\begin{equation*} 
    B(t) = \int_0^t n(s,0) \,\d s = \int_0^t S(s) \int_0^\infty n(s,a)
    \,\tau(\d a)\,\d s.
\end{equation*}
Then we will derive in \Cref{SS:mckvf} that $B$ solves the following
non-linear delay equation:
\begin{equation} \label{eq:delayClassical}
    B(t) = S_0\bigg(1-\exp\bigg(-\int_0^t \tau(a) B(t-a)\,\d a
        - \int_0^\infty \int_0^t\tau(a+s)g(a) \,\d s\,\d a
        \bigg)\bigg),
\end{equation}
where we recall that $S_0 = 1 - I_0$ is the initial fraction of susceptibles.

Our proof of \Cref{thm:ageStructure} uses a genealogical approach, where
we look backwards in time at the set of potential infectors of a focal
individual.  This approach leads to a genealogical dual to the delay equation
that we think to be of independent interest. The dual is built out of the
following branching process.

Recall that $R_0$ stands for the total mass of $\tau$ and
$\nu = \tau / R_0$. 
We define the intensity
\begin{equation} \label{eq:tauBar}
  \bar{\tau}(u) = \int_0^\infty g(a)\tau(a+u)\,\d a, \qquad u\geq 0,
\end{equation}
so that the measure $\bar{\tau}(u)\,\d u$ is the intensity measure of the infection point
process of an individual with initial age distributed as $g$. Let us further
set
\begin{equation} \label{eq:R0bar}
    \bar{R}_0 = \int_0^\infty \bar{\tau}(u)\,\d u,\quad
    \bar{\nu}(\d u) = \frac{\bar{\tau}(u)\,\d u}{\bar{R}_0}.
\end{equation}
The branching process is constructed as follows. Let us assume that individuals
in the branching process are either infected $(I)$ or susceptibles $(S)$. Suppose that the
population starts from a single $(S)$ individual. Then, at each generation,
an $(S)$ individual produces:
\begin{itemize}
    \item a $\mathrm{Poisson}\big(S_0R_0\big)$ distributed number of $(S)$
        individuals;
    \item a $\mathrm{Poisson}\big((1-S_0)\bar{R}_0\big)$ distributed number of $(I)$
        individuals.
\end{itemize}
Individuals of type $(I)$ have no offspring. Draw an oriented edge from
each individual towards its parent. Assign a length independently to
each edge, such that the length of an edge originating from an $(S)$
individual is distributed as $\nu$, and that of an edge coming from a
$(I)$ individual is distributed as $\bar{\nu}$.

The previous branching process corresponds to the large population size limit
of the set of potential infectors of a fixed individual. Type $(I)$ individuals
correspond to individuals that were initially infected. Each edge corresponds
to an infectious contact in the population, and the length of that edge is the
age of the infector when this contact occurs.

The corresponding object is a rooted \emph{geometric} tree, where edges
are endowed with a length. We define the infection path at the
root as the (a.s.\ unique) geodesic connecting the set of infectious
individuals $(I)$ to the root. Finally, the time of infection
$\sigma^\infty$ is defined as the length of the geodesic. The following
result connects the distribution of $\sigma^\infty$ to the delay
equation.

\begin{proposition}
    For any $t \ge 0$, define
    \[
        B(t) = S_0 \P( \sigma^\infty \le t).
    \]
    Then $(B(t);\, t \ge 0)$ solves the delay equation~\eqref{eq:delayClassical}.
\end{proposition}

In \Cref{SS:limitTree}, we will derive a similar dual for the delay
equation with $c \not\equiv 1$ --- see \Cref{prop:dual}.

\subsection{Link with literature}

\paragraph{Age-structured models in epidemiology.} 
The idea of considering an infection through its age structure dates back to at
least the pioneering work of Kermack and McKendrick~\cite{kermack1927}.
They introduced a general epidemic model where the infectiousness of
an individual can depend on its age of infection, which was formulated as
the solution to a delay equation equivalent to \eqref{eq:mckvf}. In the
same article \cite{kermack1927}, the authors noticed that if the
infectiousness and the recovery rate are assumed to be constant,
Equation~\eqref{eq:mckvf} reduces to a set of non-linear ODEs now known
as the SIR model. Even if the work of \cite{kermack1927} was primarily
devoted to the more general age-structured model, subsequent work on
epidemic modeling has mostly focused on extensions of the ODE special
case. Nonetheless, the original age-structured model is now receiving
renewed attention both in the mathematical literature
\cite{Reddingius1971, diekmann_limiting_1977, inaba17} and in
applications \cite{cori_new_2013, wallinga_how_2007, brauer_kermack_2005,
brauer_2012}.

In the probabilistic literature, it is only quite recently that it was
proved in \cite{barbour2013} that Equation~\eqref{eq:mckvf} describes the
large population size limit of a stochastic epidemic model similar to
ours. The setting of the main result in \cite{barbour2013} is slightly
different from that considered in the current work: the process is
assumed to be supercritical ($R_0 > 1$) and to start from a single
infected individual. After an appropriate time-shift so as to skip the
long initial branching phase when there are few infected individuals,
\cite{barbour2013} prove that the fraction of susceptibles in the
population converges to a limiting function $(S(t);\, t \in \RR)$. This
limit corresponds to the number of susceptible individuals in an
extension of \eqref{eq:mckvf} to the whole real line
\cite{diekmann_limiting_1977, Thieme1984, Thieme1985}. 
Although the law of large numbers considered in \cite{barbour2013} is
quite similar to our Theorem~\ref{thm:ageStructure}, let us outline some
important differences. 

From a purely technical point of view, \cite{barbour2013} work under
quite restrictive assumptions on the point process $\mathcal{P}$, see
Assumption~2 on the top of page~7. For instance the simple Markovian SEIR
model \cite{britton2019} does not fulfill these assumptions. Also, the
model in \cite{barbour2013} does not explicitly account for compartments
through $(X(a);\, a \ge 0)$, nor for the contact rate $(c(t);\, t \ge
0)$. These are two key modeling ingredients in the context of the
COVID-19 epidemic. While incorporating compartments would be a direct
extension of the proofs in \cite{barbour2013}, taking into account the
contact rate would raise the same serious technical difficulties as in
our work, since their proofs rely on a backward-in-time approach similar to
ours. Finally, the description of the chain of transmission events
leading to a focal infection, which is one of the main contributions of
our work, is not considered in \cite{barbour2013}.

\paragraph{Other age-structured models.} There exists a rich literature on
age-dependent population processes, not necessarily related to epidemic
models. Let us first mention the Crump--Mode--Jagers processes (CMJ),
also known as general branching processes \cite{jagers75, Taib1992}, from
which the formalism of our model is borrowed. In these processes, the
birth time of children can depend in a very general way on the age of the
parent, but individuals reproduce independently of each other. These
processes are good approximations of the early dynamics of an epidemic
when susceptible individuals are in excess. We have considered such an
approximation in an earlier work \cite{foutelrodier2020individualbased}
and proved a law of large numbers similar to Theorem~\ref{thm:ageStructure} 
in this context.

Further models have relaxed the assumption that individuals reproduce
independently by allowing the birth and death rates to depend on the
whole age distribution of the population \cite{tran2008, jagers2000,
jagers2011, hamza2013}. Under a large population size limit, the age
structure of these models converges to an extension of the
McKendrick--von~Foerster equation that generalizes \eqref{eq:mckvf}
\cite{tran2008, hamza2013}. Moreover, several central limit theorems have
been derived for this age structure \cite{tran06, fan2020}.
Although these models allow for
a very general dependence between births and the state of the population,
they require the age distribution to be a Markov process and our results
are not trivially implied by those in the above works. The techniques
used to study these models are also quite different from the
backward-in-time approach developed here. They require to see the age
structure as the solution to a stochastic equation driven by a Poisson
measure, or to use martingale techniques which could not be extended to
our setting.

\paragraph{Other non-Markovian epidemic models.}
Finally, there is a recent series of work that considers epidemic models
that are non-Markovian \cite{pang2020functional, pang2020multipatch,
forien2020epidemic, pang2021functional, pang2021functional2}, but not
structured by the age of infection. They derived laws of large numbers and
central limit theorems for extensions of the model considered in
\cite{pang2020functional} that can incorporate  
spatial heterogeneity \cite{pang2020multipatch}, varying or random
infectiosity \cite{forien2020epidemic, pang2021functional,
pang2021functional2}, and applied these models to the COVID-19 epidemic
in France \cite{forien2020estimating}.
The limiting equations that describe the
dynamics of the density of individuals in each compartments are systems
of so-called Volterra integral equations. These equations are tightly
linked to our PDE representation using the Kermack--McKendrick
equation~\eqref{eq:mckvf}, as is acknowledged explicitly in
\cite{pang2021functional2}, Proposition~2.1. All the models with
non-vanishing immunity that they consider (SIR, SEIR) can be formulated
within our framework. The infection point process $\mathcal{P}$ is
obtained by starting either a homogeneous Poisson point process
\cite{pang2020functional}, an inhomogeneous Poisson point process
\cite{forien2020epidemic, pang2021functional}, or an inhomogeneous
Poisson point process with random intensity \cite{pang2021functional2},
at age $a=0$ in the SIR case, or after an exposed phase in the SEIR case.
Moreover, the proof techniques they use rely heavily on a representation
of their model as the solution to a stochastic equation driven by a
Poisson measure, which does not hold in our more general setting.
Nevertheless, let us acknowledge that their techniques allow to derive
central limit theorems for their models.

\subsection{Further discussion}

Let us discuss some practical implications and limitations of our results
for epidemiological applications, as well as some avenues for future
work.

\paragraph{Impact of the initial condition.} A major limitation to the
practical interest of Theorem~\ref{thm:ageStructure} is that the age
structure of the initial population should converge to a known limit,
for which a positive fraction of individuals are infected. This means
that our result could only be applied once the epidemic has been
spreading for a long enough time, and that the initial age structure of
the population needs to be prescribed. In practice this age structure can
hardly be estimated.

In applications, the large number of individuals observed at $t=0$
results from the growth of the epidemic out of a few initial individuals.
It is thus natural to try to derive a law of large numbers similar to
Theorem~\ref{thm:ageStructure} but started from a few infected
individuals. Such a result was already derived in \cite{barbour2013}, for
$c \equiv 1$ and under some additional technical assumptions. It
was shown that the limit of the age structure then converges to an
extension of \eqref{eq:mckvf} to the whole real line $t \in \RR$. This 
extension is unique up to a shift \cite{diekmann_limiting_1977}, and does
not depend on any initial age structure, solving the above issue. We
expect that a similar result holds in our setting with $c \not\equiv 1$.
However, in this case, the solutions to \eqref{eq:mckvf} on the real line
are neither unique nor shift-invariant. It is a more delicate issue to
describe these solutions, and to understand which one of them is selected
by the initial randomness of the stochastic process. This relates to
existing work on dynamical systems perturbed by a small noise and started
near an unstable equilibrium \cite{barbour2016, baker2018}.

\paragraph{Speed of convergence and deviations.} 
Theorem~\ref{thm:ageStructure} provides a rigorous justification to the
use of deterministic age-structured epidemic models, as limits of a large
class of stochastic individual-based models. For the purpose of
applications, it would be desirable to understand quantitatively how
accurate this approximation is, that is, to derive a speed of
convergence of the stochastic model to its deterministic limit. An even
more important question for statistical applications would be to
characterize the deviation of the stochastic system from the limit.
We have derived our law of large numbers under a minimal first moment
assumption \eqref{eq:intensity} on the infection point process. We expect
that a central limit theorem similar to those obtained in
\cite{pang2020functional, fan2020, tran06} should hold under a second
moment assumption on $\mathcal{P}$. This would entail that the fraction
of individuals in the various compartments remains at a distance of order
$1/\sqrt{N}$ from the deterministic limit, and would provide a natural
limiting expression for the likelihood of the process. Note that, since
we do not assume that $\mathcal{P}$ is a Poisson or a Cox process, the
correlation structure of the limiting Gaussian process should be
different from than in \cite{pang2020functional, fan2020, tran06} and
more similar to the co-variance structure of a Crump--Mode--Jagers
process, described for instance in \cite{Jagers1984a}.

\subsection{Outline}

The rest of this paper is organized as follows. A formal description 
of the model is provided in \Cref{SS:descriptionModel}, and the
Kermack--McKendrick PDE is studied in \Cref{SS:mckvf}.

\Cref{sec:graph} contains the graph construction of our model, as well as
a rigorous definition of the ancestral process mentioned in
\Cref{SS:contactTracing}. Our proofs rely on showing the local weak convergence of
the graph of potential infectors to a limiting Poisson tree.
\Cref{SS:limitTree} describes this limiting tree, and provides a
characterization of the transmission chain leading to the infection of
individual in terms of the $h$-transform of a renewal process. Finally,
the convergence to the Poisson tree is carried out in \Cref{SS:proofLLN}
and our law of large numbers is proved in \Cref{SS:proofMainThm}.

\section{Description of the model}
\label{SS:descriptionModel}

In the following, we will consider an epidemic model in which
individuals' life trajectories are represented by independent stochastic
processes. We distinguish between two types of individuals:
\begin{itemize} 
    \item Susceptible individuals that have never been infected before.
    \item Infected individuals that have been infected in the past. We
        emphasize that the meaning of infected is a bit broader than
        usual. For instance, a recovered or dead individual is considered
        as infected. To each infected individual, we associate an age. The age
        is the time elapsed since the beginning of the infection.
\end{itemize}
There are $N$ individuals in the population. To each individual 
$x \in [N]$, we associate a pair of processes $(\mathcal{P}_x, X_x)$
describing respectively the process of secondary infections
and the successive stages of the disease experienced by the focal
individual $x$. More precisely:
\begin{itemize}
    \item The \emph{life-cycle process}, denoted by $(X_x(a);\, a \geq 0)$,
        is a random process valued in $\mathcal{S}$ where $X_x(a)$ is the stage
        of the disease (e.g., exposed, death, etc.) of $x$ at age $a$.
    \item The \emph{infection point process} $\mathcal{P}_x$ is a point
        measure describing the ages of potential infections.
\end{itemize}

Let us denote by $\mathcal{X}_x \coloneqq (\mathcal{P}_x, X_x)$. We will always
assume that $(\mathcal{X}_x;\, x \in [N])$ are i.i.d.\ and denote by 
$\mathcal{X} = (\mathcal{P}, X)$ their common distribution. The state
space of $\mathcal{X}$ is denoted by $\mathscr{X}$.

\begin{remark}
    Note that we allow for non-trivial dependence between the life-cycle and
    the infection process. Examples of such dependence can be that 
    a deceased individual is not infectious anymore, a hospitalized individual
    may have a reduced potential of infection due to quarantine, etc.
\end{remark}

We suppose that at $t = 0$, every individual is independently infected
with probability $I_0$. Let ${\cal I}_0^N$ be the set of initially
infected individuals. For each $x \in \mathcal{I}^N_0$ we need to
prescribe an age, or equivalently, an infection time. We assume that,
conditional on $\mathcal{I}^N_0$, the ages of the initial individuals
$(Z_x;\, x \in \mathcal{I}^N_0)$ are i.i.d.\ with common distribution
$g$. We also set $Z_x = 0$ for $x \notin \mathcal{I}_0$ for convenience.
Let us denote by $(\sigma^N_x;\, x \in \mathcal{I}^N_0)$ the infection
time of the initial infected, that is, $\sigma^N_x = -Z_x$. 

The epidemic now spreads as follows. Suppose that, at some time $t_0$, we have
defined a set $\mathcal{I}^N_{t_0} \subseteq [N]$ of infected individuals at time
$t_0$, and a vector $(\sigma^N_x;\, x \in \mathcal{I}^N_{t_0})$ of infection times.
Let $t_1$ be the first atom after $t_0$ of the point measure
\[
    \sum_{x \in \mathcal{I}^N_{t_0}} \sum_{a \in \mathcal{P}_x} \delta(\sigma^N_x + a).
\]
If there is no such atom, the infection stops. Otherwise, let $U$ be
uniformly chosen in $[N]$, independent of the rest: it is the first
individual that comes in contact with any of the infected individuals 
after time $t_0$. If $U \in \mathcal{I}^N_{t_0}$, then nothing happens, and 
we carry out the same procedure for the next atom $t_2$. If $U \notin
\mathcal{I}^N_{t_0}$, then, with probability $1-c(t_1)$, the infection is
ineffective in which case nothing happens and we consider the next infection
time $t_2$. Otherwise, set $\mathcal{I}^N_{t_1} = \mathcal{I}^N_{t_0} \cup \{U\}$
and $\sigma^N_U = t_1$, and continue the procedure as if starting from time $t_1$
with the initial infected set $\mathcal{I}^N_{t_1}$.
This inductive procedure will be reformulated in terms of an infection graph in
\Cref{SS:infectionGraph}.

\section{Kermack--McKendrick PDE} 
\label{SS:mckvf}

In this section we provide our definition of the solution to the
Kermack--McKendrick equation~\eqref{eq:mckvf}. We start with 
a formal resolution of the PDE using the method of characteristics.

Let $I_0$ be the initial density of infected individuals and $g$ the initial
age profile of the population. First, note that if $n$ is solution of the PDE,
then for every pair $(t,a)$ of non-negative numbers, $s \mapsto n(t-s, a-s)$ is
constant on $(0,t\wedge a)$. This yields
\begin{equation} \label{eq:mild-sir}
    \forall t,a \ge 0,\quad
    n(t,a) = \begin{cases}
        I_0 g(a-t) &\text{ when } a > t\\
        b(t-a) &\text{ when } a \leq t,
    \end{cases}
\end{equation} 
with 
\[
    \forall t \ge 0,\quad b(t) \coloneqq n(t, 0)
\]
is the number of new infections at time $t$. Moreover, 
\begin{align*}
    \dot{S}(t) &= -\int_0^\infty \partial_t n(t,a) \,\d a 
               = \int_0^\infty \partial_a n(t,a) \,\d a \\
               &= -b(t)
                = - c(t) S(t) \int_0^\infty \tau(a) n(t, a)\,\d a.
\end{align*}
As a result, since $S(0) = S_0$, we have
\begin{align*}
  S(t) &= S_0\exp\bigg(-\int_0^t c(s) \int_0^\infty \tau(a) n(s, a)\,\d a\,\d s\bigg)\\
       &= S_0 \exp\bigg(
           - \int_0^t c(s) \big(\int_0^s\tau(a)b(s-a)\,\d a 
           + I_0\int_s^\infty \tau(a) g(a-s) \,\d a\big)\,\d s
       \bigg)\\
       &= S_0 \exp\bigg(
           - \int_0^t c(s) \big(\int_0^s\tau(s-a)b(a)\,\d a 
           + I_0\bar{\tau}(s)\big)\,\d s,
       \bigg)
\end{align*}
where $\bar{\tau}(s)$ was defined in \eqref{eq:tauBar},
so necessarily 
\[
    B(t) \coloneqq \int_0^tb(s)\,\d s = S_0-S(t)   
\]
solves the nonlinear delay equation
\begin{equation} \label{eq:delay}
    B(t) = S_0 \Big[ 
        1 - \exp\bigg(-\int_0^t c(s) \big(\int_0^s\tau(s-a)\,B(\d a)
                      +I_0\bar{\tau}(s)\big)\,\d s\bigg)
        \Big]
\end{equation}
where $B(\d a) = b(a) \,\d a$ is the Stieltjes measure associated to the
nondecreasing map $B$. This motivates the following definition of a solution
to the Kermack--McKendrick equation.

\begin{definition}
    We say that $(n(t,a);\, t,a \ge 0)$ is a weak solution to
    \eqref{eq:mckvf} if there exists a nonnegative function $(b(t);\, t\ge 0)$
    such that:
    \begin{enumerate}
        \item the functions $n$ and $b$ are related through \eqref{eq:mild-sir};
        \item the function $B(t) \coloneqq \int_0^tb(s)\,\d s$ solves the 
            delay equation~\eqref{eq:delay}.  \qedhere
    \end{enumerate}
\end{definition}

If a nondecreasing function $B$ satisfies \eqref{eq:delay}, then we have the
following inequality:
\begin{equation} \label{eq:absoluteContinuity}
    B(t+u) - B(t) \le S(0) \int_t^{t+u} c(s) \big( \int_0^s \tau(s-a) B(\d a)
    + I_0 \int_0^\infty \tau(a+s) g(a) \,\d a \big) \,\d s.
\end{equation}
The previous inequality readily entails that $B$ is absolutely continuous,
and thus that we can find $b$ such that $B(t) = \int_0^t b(s)\,\d s$.
Therefore, existence and uniqueness of solutions to \eqref{eq:mckvf} 
reduce to existence and uniqueness of nondecreasing solutions to
\eqref{eq:delay}, which is provided by the following result.

\begin{lemma}
    There is a unique nondecreasing, nonnegative solution to
    \eqref{eq:delay}.
\end{lemma}

\begin{proof}
    Let us denote by $E$ the set of all nondecreasing, nonnegative,
    c\`adl\`ag functions on $[0, \infty)$. Recall the definition of the
    Malthusian parameter $\alpha$ from \eqref{eq:malthusian}. Fix
    some for $\gamma > \alpha \vee 0$, so that we have
    \begin{equation} \label{eq:largerThanMalthus}
        \int_0^\infty e^{-\gamma t} \tau(t) \,\diff t < 1.
    \end{equation}
    Define
    \[
        E_\gamma = \{f \in E : \int_0^\infty e^{-\gamma t} f(t) \,\d t <
        \infty \}.
    \]
    We endow $E_\gamma$ with the metric
    \[
        d_\gamma(f,g) = \int_0^\infty e^{-\gamma t} \abs{f(t)-g(t)} \,\diff t
    \]
    which makes $(E_\gamma, d_\gamma)$ a complete metric space.
    As any solution to \eqref{eq:delay} is bounded and continuous, it is
    sufficient to show existence and uniqueness of the solution in
    $E_\gamma$.

    We introduce the operator $\Phi \colon E_\gamma \to E_\gamma$ such that
    \begin{multline*}
        \Phi f(t) = 
        S_0 \Big(1- \exp\Big( 
            -\int_0^t c(s) \big(\int_0^s \tau(s-a) f(\d a) \big) \,\d s \\
            - I_0\int_0^\infty \big(\int_0^t c(s) \tau(a+s)\,\d s\big) g(a) \,\d a
        \Big) \Big),
    \end{multline*}
    where $f(\diff a)$ denotes the Stieltjes measure associated to $f$.
    Note that $\Phi f \in E_\gamma$, since it is clear that $\Phi f$ is
    bounded, continuous, nonnegative and nondecreasing. Let us show that 
    $\Phi$ is a contraction. We have, for $f_1, f_2 \in E_\gamma$, 
    \begin{align*}
        d_\gamma(\Phi f_1, \Phi f_2) 
        &\le S_0 \int_0^\infty e^{-\gamma t} 
        \abs*{\int_0^t c(s)\big( \int_0^s \tau(s-a) f_1(\d a) - 
        \int_0^s \tau(s-a) f_2(\d a) \big) \,\d s } \,\d t  \\
        &\le \int_0^\infty e^{-\gamma t} 
        \big( \int_0^t \tau(s) \abs*{f_1(t-s) - f_2(t-s)} \,\d s\big) \,\d t \\
        &= d_\gamma(f_1, f_2) \int_0^\infty e^{-\gamma t} \tau(t) \,\d t.
    \end{align*}
    By \eqref{eq:largerThanMalthus} we know that $\int_0^\infty e^{-\gamma t}
    \tau(t) \,\d t < 1$, showing that $\Phi$ is a contraction. The Banach
    fixed point theorem therefore shows that there exists a unique 
    $B \in E_\gamma$ such that $\Phi B = B$, ending the proof.
\end{proof}

\section{Graph of infection} \label{sec:graph}

\subsection{Infection graph} 
\label{SS:infectionGraph}

Recall the infection model defined in \Cref{SS:descriptionModel}, and the
notation $(\mathcal{P}_x;\, x \in [N])$ for the infection point
processes. Recall that $(Z_x;\, x \in [N])$ are i.i.d.\ random variables
with law
\begin{equation} \label{eq:initialAge}
    Z_x \sim \delta_0(\diff a) (1- I_0) \ + \  I_0 g(a) \diff a,
\end{equation}
and that we have defined $\mathcal{I}^N_0 = \{ x : Z_x > 0 \}$ for the
set of initially infected individuals, and $(\sigma^N_x = -Z_x;\, x \in
\mathcal{I}^N_0)$ for their infection (or birth) time. Intuitively, $Z_x$
encodes the age of infection of individual $x$ at time $t=0$. Susceptible
individuals have age $0$, whereas the age of an infected individual is
chosen according to the density $g$. Define the shifted infection measure
\begin{equation} \label{eq:shiftedPP}
    \widehat {\cal P}_x \ = 
    \sum_{a \in \mathcal{P}_x}  \indic{a - Z_x \ge 0} \delta(a-Z_x)  
\end{equation}
Note that if $x$ is susceptible (i.e., $Z_{x}=0$), we have
$\widehat{\mathcal{P}}_x = {\cal P}_x$. Vertices with $Z_x=0$ will be
said of type susceptible $(S)$. Vertices with $Z_{x}>0$ will be said of
type infected $(I)$.

Recall that each atom of a point process $\widehat{\mathcal{P}}_x$
encodes a potential infectious contact, which is targeted to a uniformly
chosen individual in the population. We enrich the infection point
processes by adding the information about the label of the target
individual. Formally, we consider a sequence of i.i.d.\ random variables
$(U_{x,i};\, x\in[N], i\in\mathbb{N})$ uniformly distributed on $[N]$.
Define 
\begin{equation} \label{eq:decoratedPP}
    \forall x\in[N], \quad \widehat{\mathcal{N}}_{x} 
    \ \coloneqq \ 
    \sum_{a_i \in\widehat{\mathcal{P}}_{x}} \delta(a_i, U_{x,i})
\end{equation}
where $a_1<a_2 <\cdots$ in the sum are the atoms of
$\widehat{\mathcal{P}}_x$ listed in increasing order. We now build a
graph out of the family $(\widehat{\mathcal{N}}_x;\, x \in [N])$ that
records the potential infections in the population.

\begin{definition}[Infection graph]\label{def:w-graph}
    The \emph{infection graph} built from the i.i.d.\ collection
    of triplets $(\widehat{\mathcal{N}}_x, X_x, Z_x;\, x \in [N])$ is the random
    oriented geometric marked graph $\mathcal{G}^N=(V^N,E^N)$ with $V^N=[N]$ and
    \[
        E^N = \bigcup_{x \in [N]} \bigsqcup_{(a,j) \in \widehat{\mathcal{N}}_x} \{(i,j)\},
    \]
    where the second union is a disjoint union, meaning that for each
    pair $(i,j)$ we allow for multiple edges from $i$ to $j$ in the
    infection graph. The marks and edge lengths are defined as follows.
    \begin{enumerate}
    \item Each edge $e=(i_e,j_e)$ corresponds to an
    atom $(a_e, j_e)$ of the point process $\widehat{\mathcal{N}}_{i_e}$
    defined in \eqref{eq:decoratedPP}. The age $a_e$ will be referred to
    as the length of edge $e$.
    \item For each vertex $x\in V^N$, we define the mark at $x$ as
    \[ 
        m_x \coloneqq (Z_x, \mathcal{X}_x),
    \]
    where the initial infection age $Z_x$ is defined in
    \eqref{eq:initialAge}.
    \end{enumerate}
\end{definition}

\begin{remark}
    As stated in the theorem, ${\cal G}^N$ is an oriented geometric marked
    graph. By geometric, we mean that edges have lengths. The orientation is
    dictated by the direction of potential infections, and the meaning of an
    edge $(i,j)$ is that individual $j$ is potentially infected by $i$.
    Finally, the first coordinate of the marking allows to distinguish
    between infected and susceptible individuals at time $t=0$.
\end{remark}

A path in $\mathcal{G}^N$ is a set of edges $\pi = (e_1, \dots, e_n)$ such
that, $j_{e_k} = i_{e_{k+1}}$, with the notation $(i_e, j_e)$ for the origin
and target vertices of the edge $e$. The length of a path $\abs{\pi}$ is
defined as 
\[
    \abs{\pi} = \sum_{e \in \pi} a_{e}.
\]
The genealogical (or topological) distance is defined as the number of
edges composing the path ($n$ in our specific example). For $k \le n$, we
define the $k$-truncation of $\pi$ as 
\[ 
    \tau_k \pi \ \coloneqq \ (e_1, \dots, e_k).
\]
We say that $\pi$ is a path from $i$ to $j$ if $i_{e_1} = i$ and
$j_{e_n}=j$. A path in $\mathcal{G}^N$ from $i$ to $j$ corresponds to a
potential infection chain between $i$ and $j$. The length of the path is
the time elapsed between the infection of $i$ and $j$. The genealogical
distance is the number of infectors along the chain.

It turns out the infection graph that we have constructed corresponds to
a directed version of a configuration model. The (undirected)
configuration model is a well-studied random graph model where, starting
from a prescribed number of half-edges for each of the $N$ vertices
$(D_1, \dots, D_N)$, a random graph is obtained by pairing these
half-edges uniformly at random \cite[Section~2.2.2]{van2017stochastic}.
Let us make this connection explicit. 

For $x \in [N]$, denote by $D^\Out_x$ (resp.\ $D^\In_x$) the out-degree
(resp.\ in-degree) of vertex $x$ in $\mathcal{G}^N$. Clearly, $D^\Out_x =
\abs{\widehat{\mathcal{P}}_x}$ and, since each out-edge is pointing
towards a uniformly chosen vertex, conditional on $(D^\Out_x)_x$,
\begin{equation} \label{eq:multinomial}
    (D^\In_1, \dots, D^\In_N) \sim \mathrm{Multinomial}\big(
        \sum_{x \in [N]} D^\Out_x ; (\tfrac{1}{N}, \dots, \tfrac{1}{N})
    \big).
\end{equation}
Suppose that $(D^\Out_x, D^\In_x)_x$ are prescribed. We can construct a
graph with this given sequence of in and out degrees in the following
way:
\begin{enumerate}
    \item attach to each vertex $x \in [N]$ a number $D^\Out_x$ of out
        half-edges, and a number $D^\In_x$ of in half-edges; and
    \item pair each out half-edge to a different in half-edge.
\end{enumerate}
If the pairing in the second step is made uniformly among all
possibilities, the resulting random graph is called a \emph{directed
configuration model} with degree sequence $(D^\Out_x, D^\In_x)_x$.

In the infection graph it is not hard to see that, again because each
out-edge is pointing towards a uniformly chosen vertex, conditional on 
$(D^\Out_x, D^\In_x)_x$ the pairing of in- and out-edges in
$\mathcal{G}^N$ is made uniformly. Furthermore, conditional on
$(D^\Out_x, D^\In_x)_x$, the marks $(m_x)_x$ are independent, and $m_x$
has the distribution of $(Z, \mathcal{X})$ conditional on
$\abs{\widehat{\mathcal{P}}_x} = D^\Out_x$. We record this connection as
a proposition for later use.

\begin{proposition} \label{prop:configurationModel}
    Let $(V^N, E^N)$ be a directed configuration model on $N$ vertices
    with degree sequence $(D^\Out_x, D^\In_x)_x$ where $(D^\Out_x)_x$ are
    i.i.d.\ and distributed as $\abs{\widehat{\mathcal{P}}}$ and
    $(D^\In_x)_x$ are distributed as \eqref{eq:multinomial}. Conditional
    on $(V^N, E^N)$, mark the vertices and add edge lengths in such a way
    that:
    \begin{enumerate}
        \item the marks $(m_x)_x$ are independent and 
            \[
                m_x \sim (Z, \mathcal{P}, X) \text{ conditional on
                    $\abs{\widehat{\mathcal{P}}} = D^\Out_x$};
            \]
        \item if $(e_i)$ are the edges going out of $x$, then $a_{e_i}
            = A_i$, where $(A_i;\, i \le \abs{\widehat{\mathcal{P}}_x})$
            is a uniform permutation of the atoms of $\widehat{\mathcal{P}}_x$.
    \end{enumerate}
    Then $(V^N, E^N, (m_x)_x, (a_e)_e)$ is distributed as the infection
    graph $\mathcal{G}^N$.
\end{proposition}

\subsection{Infection process}
\label{sec:def:active-goedesics}

Conditional on a realization of the infection graph $(V^N, E^N)$, we attach an
additional independent random variable $s_e$  uniform on $[0,1]$ to every
edge $e\in E^N$ of the graph. This random variable  will encode what we
will call the contact intensity of edge $e$. Roughly speaking, if the
contact occurs at time $t$, this contact translates into an
infection if{f} two conditions are satisfied. First, the contact
intensity should be strong enough in the sense that  $s_e \leq c(t)$ (see
\eqref{eq:r1} below). Secondly, the target individual should not have
been infected before (see \eqref{eq:r2} below). We make this more precise
in the next definition; see also \Cref{fig:active_paths}.

\begin{figure}[ht] \centering
  \includegraphics{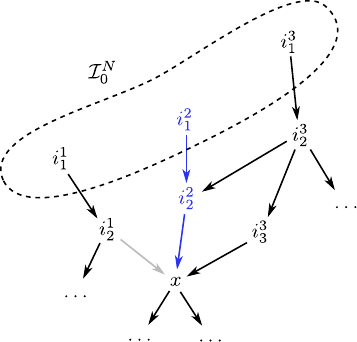}
  \caption{Ancestors of individual $x$ in $\mathcal{G}^N$.
    In this picture, we have four potential infection paths from $\mathcal{I}_0^N$ to $x$: $\pi^1=(e^1_1,e^1_2)$, $\pi^2=(e^2_1,e^2_2)$, $\pi^3=(e^3_1,e^3_2, e^3_3)$ and $\pi^4=(e^3_1,e_*, e^2_2)$, where we write $e^\ell_k = (i^\ell_k,i^\ell_{k+1})$ and $e_* = (i^3_2,i^2_2)$.
    Assume first that $\pi^1$ is the shortest path, but that $s_{e^1_2}>c(|\pi^1|)$ --- the edge is grayed out in the figure.
    Then $\pi^1$ is not an active path.
    Now let us assume that $\pi^2$ is active and that $|\tau_1\pi^2|=a_{e^2_1} < a_{e^3_1}+a_{e_*}=|\tau_2\pi^4|$.
    This means that $\pi^4$ cannot be the active geodesic.
    Finally, if $\pi^2$ and $\pi^3$ are two active paths and $|\pi^2| < |\pi^3|$, then $\pi^2$ (in blue) is the active geodesic from $\mathcal{I}_0^N$ and $\sigma^N_x = |\pi^2|$.}
  \label{fig:active_paths}
\end{figure}

\begin{definition}[Active geodesic] \label{def:active-goedesics}
Let $\pi=(e_1,\cdots,e_n)$ be a path with $i_{e_1}\in{\cal I}_0^N$. The path is said to be active if{f}
\begin{equation}\label{eq:r1}
    \forall k \le n, \ \ s_{e_k}  \leq c\big( |\tau_k  \pi| \big).
\end{equation}
For every $x\notin{\cal I}_0^N$, let $\Xi^N(x)$ be the set of active
paths from ${\cal I}_{0}^N$ to $x$. The path is said to be the active
geodesic from ${\cal I}_0^N$ to $x$  if{f} 
\begin{equation}\label{eq:r2}
    \forall k\le n, \ \ \tau_k\pi  \ = \ \argmin_{\pi' \in \Xi^N(j_{e_k})} |\pi'|. 
\end{equation}
Finally, we define the infection time of $x$ --- denoted by $\sigma^N_x$
--- as the length of the active geodesic from ${\cal I}_0^N$ to $x$, with
the convention that $\sigma^N_x=\infty$ if the geodesic does not exist. 
\end{definition}

\begin{remark}
\begin{enumerate}
\item Since $\tau$ has a density, there is at most one path satisfying
    the minimization problem \eqref{eq:r2}. 
\item If $c \equiv 1$, then any path in the infection graph is active, so
    that our definition coincides with the usual definition of a geodesic
    on a geometric graph. In particular, \eqref{eq:r2} just states that if
    $\pi=(e_1,\dots,e_n)$ is the geodesic from ${\cal I}_0^N$ to $x$,
    then the truncated path $\tau_k \pi$ is the geodesic from ${\cal I}_0^N$ 
    to $j_{e_k}$. Thus, when $c\equiv 1$, all the information about the
    infection process is contained in the infection graph and the extra
    variables $s_e$ do not play any role.
\end{enumerate}
\end{remark}

\subsection{The ancestral path}\label{sect:ancestral-path}

\begin{definition}[Infection and ancestral paths]\label{def:infection-path} 
\leavevmode
\begin{itemize}
    \item Let us consider $x$ of type  $(S)$ such that $\sigma^N_x<\infty$ and write $\pi=(e_1,\dots,e_n)$ with $e_k=(i_k,j_k)$
    for the active geodesic from ${\cal I}_0^N$ to $x$.
    We define the infection path ${\cal R}^N_x$ from $x$ to ${\cal I}_0^N$ as
    \begin{equation}\label{eq:active-geo}
        {\cal R}^N_x(0) =\sigma^N_x, \ \  \forall \ell\in[n], \ \ {\cal R}^N_x(\ell) \ = \ \sigma^N_x- \sum_{k=0}^{\ell-1} (a_{e_{n-k}} + Z_{i_{{n-k}}}).
    \end{equation}
    Finally, we define the ancestral process as
    \[
        {\cal A}^N_x \ \coloneqq \ \bigg( {\cal R}^N_x(\ell),  {\cal X}_{v_{\ell}}   \bigg)_{\ell=0}^n, \ \ 
        \mbox{where $v_k = \left\{ \begin{array}{cc} i_{n+1-\ell} & \mbox{if $\ell\neq 0$} \\ x=j_n & \mbox{if $\ell=0$} \end{array}\right.  $} 
    \]
    to be the vector recording the information along the chain of infection (age of infection, infection measure, life-cycle).
    \item If $x$ is of type $(S)$ but $\sigma^N_x=\infty$, then ${\cal A}_x^N$ is defined as the empty sequence. 
    \item If $x$ is of type $(I)$, then 
    ${\cal A}^N_x \coloneqq ( \sigma_x^N, {\cal X}_x )$.
\end{itemize}
\end{definition}
In words, the random path ${\cal R}^N_x$ is obtained by tracing backward
in time the chain of infection from the focal individual $x$ to the set
of initially infected individuals. The increments of the path are given by the
successive age of infection. ${\cal R}^N_x(\ell)$ is the time of infection
of the $\ell$-th ancestor along the chain;  the variable ${\cal X}_{v_{\ell}}$ 
encodes its infection and life-cycle processes. (Assuming that
individuals are ranked from the focal individual $x$ to the initial
$(I)$ individual.) Note that in the sum \eqref{eq:active-geo}, all the
$Z_{i_{n-k}}$ terms are equal to $0$, except when $k=n-1$. In words, the
oldest ancestor is the only type $(I)$ individual along the chain of
infection. For $k=n-1$, the corresponding increment is decomposed into
two parts: (1) the undershoot $a_{e_1}$ and (2) the overshoot $Z_{i_{1}}$
corresponding to the age of infection of the $(I)$ individual at time
$t=0$. See \Cref{fig:ancestral_path}.

\begin{figure}[ht] \centering
  \includegraphics{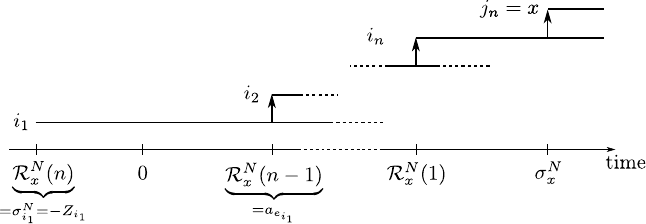}
  \caption{Infection path of individual $x$.} \label{fig:ancestral_path}
\end{figure}

\bigskip

We close this section by a brief description of the topology underlying
the set of ancestral paths. Let ${\cal M}$ refer to the set of locally
finite positive measures on $[0, \infty)$, and equip $\mathcal{M}$ with a
metric $d_{\mathcal{M}}$ that induces the vague topology
\cite[Section~4.1]{Kal17}. We denote by $d_{\cal S}$ the Skorohod metric on
the set of càdlàg processes valued in ${\cal S}$ denoted by
$\mathbb{D}(\mathcal{S})$. The space $\RR_+ \times
{\cal M} \times \mathbb{D}(\mathcal{S})$ is equipped with the sup metric $\rho$ defined
as
\[
\forall (x,y,z), (x',y',z'), \ \ \rho\left( (x,y,z), (x',y',z') \right) \ = \ \max\bigg( |x-x'|, d_{\cal M}( y,y'), d_{\cal S}( z,z')  \bigg).
\]
Each ancestral path is valued in the space
\[
\bigcup_{n=0}^\infty \bigg( \RR_+ \times {\cal M} \times \mathbb{D}(\mathcal{S}) \bigg)^n
\]
equipped with the metric ${\cal D}$ defined as follows
\[
    {\cal D}\bigg((m_1,\dots, m_{n}), (m'_1,\dots, m_{n'}')\bigg)  = 
    \begin{cases}
        1 &\text{if $n\neq n'$} \\
        1 \wedge \max( \rho(m_1,m_1'),\dots, \rho(m_n, m'_{n'}) ) &\text{if $n=n'$}.
    \end{cases}
\]

\subsection{Local weak convergence}
\label{SS:localTopology}

We introduce the notion of \emph{local weak convergence} \cite{AS04, BS01}. Intuitively, a
sequence of graphs converges in the local weak sense if the local
structure around a typical vertex (meaning a uniformly chosen vertex)
converges in distribution to a random limit. We make this definition precise.

A pointed oriented geometric marked (pogm) graph $G$ is characterized by
five coordinates $G = (V,E,(a_e),(m_x), \emptyset)$, respectively the set of
vertices, the set of edges, $(a_e)_{e\in E}$ the lengths of the edges,
$(m_x)_{x\in V}$ the set of marks, and $\emptyset \in V$ the pointed
vertex. We let $\mathscr{H}$ denote the set of pogm graphs, and equip it
with a metric $d_{\mathscr{H}}$ so that $(\mathscr{H}, d_{\mathscr{H}})$ 
is a Polish space. A graph isomorphism $\phi$ between two \emph{finite}
pogm graphs $G=(V,E,(a_e),(m_x),\emptyset)$ and
$G'=(V',E',(a'_e),(m'_x),\emptyset')$ is a bijection from $V$
to $V'$ such that
\begin{enumerate}
    \item $(u,v)\in E$ if{f} $(\phi(u), \phi(v))\in E'$. 
    \item $\phi$ maps the  reference vertex of $G$ to the
        reference vertex in $G'$. 
\end{enumerate}
By convention, we set $\min(\emptyset)=\infty$ in the following. Let
$G = (V,E, (a_e),(m_x),\emptyset)$, $G'=(V',E',(a'_e),(m'_x),\emptyset')$ 
be two elements of $\mathscr{H}$. Define 
\[
    d(G, G') = \min\{1, \min_{\phi}
        \bigg(\max_{e\in E} \abs{a_e - a'_{\phi(e)}} \vee \max_{x\in
V} | \rho(m_x, m'_{\phi(x)}) | \bigg) \}
\]
where the minimum is taken over all possible graph isomorphisms between
the two graphs (in the sense prescribed above, that is, we only consider
the isomorphisms preserving the pointed vertex). If there is no such
isomorphism between $G$ and $G'$, we set $d(G, G')=1$.  

For $G \in \mathscr{H}$ and $y \in G$, the topological (or
\emph{genealogical}) distance to the reference vertex $x$ is defined as
\[  
    \inf\{ n  : \text{there exists a path $(y=x_0,\dots, x_n=x)$  in $G$}\}. 
\]
For every $r\in \mathbb{N}^*$, we denote by $[G]_r$, the
subgraph induced by the vertices at a topological distance to the
origin, that is, to the pointed vertex, less than $r$. For two elements
$G, G'\in\mathscr{H}$, we define the (pseudo-)distance
$d_\mathscr{H}$ as follows
\[ 
    d_\mathscr{H}(G, G') \ = \  \sum_{r} 2^{-r}
    d([G]_r, [G']_r). 
\]
The metric $d_\mathscr{H}$ naturally induces a notion of local
convergence on (equivalence classes of) $\mathscr{H}$. Using standard
arguments, we can see that $(\mathscr{H}, d_\mathscr{H})$ is a Polish
space.

Given an oriented geometric marked graph $G^N$ of size $N$, for $x \in [N]$
define $G^N(x)$ as the subgraph of $G^N$ induced by all vertices $y$ with an
oriented \emph{valid} path from $y$ to $x$ (including $x$ itself), 
where we call a path $(y = y_0, y_1, \dots , y_k, x)$ valid if and only
if for all $1 \le i \le k$, the node $y_i$ is of type ($S$). $G^N(x)$ is
therefore the graph that contains all potential chains of infection
leading to the infection of node $x$ from an initially infected
individual.

We treat $G^N(x)$ as an element of $\mathscr{H}$, with $x$ as the
reference vertex. We can construct a measure on $\mathscr{H}$ out of the
graph $G^N$ by assigning the root $x$ uniformly at random:
\[
    P(G^N) = \frac{1}{N} \sum_{x \in [N]} \delta_{G^N(x)}.
\]
If the graph $G^N$ is random, $P(G^N)$ is a \emph{random} measure. 
The following definition is taken from Definition~3.6 in
\cite{garavaglia20}.

\begin{definition}[Local weak convergence]
    We say that a sequence of random pogm graphs $(G^N)_N$ converges in
    probability in the local weak sense to a random graph $G \in
    \mathscr{H}$ if 
    \[
        P(G^N) \longrightarrow \mathcal{L}(G)
    \]
    in probability for the weak topology on measures on $\mathscr{H}$,
    and where $\mathcal{L}(G)$ is the law of $G$.
\end{definition}

We end this section with a direct consequence of the various definitions.

\begin{lemma} \label{lem:double-continuous-mapping}
    Consider a metric space $E$ and a functional $\Phi \colon \mathscr{H}
    \to E$. Suppose that for all $N\geq 1$, $G^N$ is a random pogm graph of size $N$, and
    that $(G^N)$ converges in probability in the local weak sense to 
    some other pogm graph $G$. If $\Phi$ is continuous on
    a set $A \subseteq \mathscr{H}$ such that $\P(G \in A) = 1$, then
    \[
        \frac{1}{N} \sum_{i=1}^N \delta_{\Phi(G^N)} \longrightarrow
        \mathcal{L}(\Phi(G))
    \]
    in probability for the weak topology on measures on $E$, and where 
    $\mathcal{L}(\Phi(G))$ is the law of $\Phi(G)$.
\end{lemma}

\begin{proof}
    For a probability measure $P$ on $\mathscr{H}$, let $P \circ
    \Phi^{-1}$ denote the push-forward measure of $P$ by $\Phi$. Clearly
    \[
        \frac{1}{N} \sum_{i=1}^N \delta_{\Phi(G^N)} 
        = P(G^N) \circ \Phi^{-1}.
    \]
    According to the continuous mapping theorem, the result is proved if
    we can show that the mapping $P \mapsto P \circ \Phi^{-1}$ is
    continuous at $P = \mathcal{L}(G)$ for the weak topology. Let $P^N
    \to \mathcal{L}(G)$ weakly, and $\tilde{G}^N \sim P^N$. Another
    application of the continuous mapping theorem shows that
    $\Phi(\tilde{G}^N) \to \Phi(G)$ in distribution, showing the result.
\end{proof}

\section{A limiting Poisson random tree} 
\label{SS:limitTree}

\subsection{Palm infection measures}

Recall that ${\cal X}_x = ({\cal P}_x, X_x)$ is the pair encoding the
infection and the life-cycle process and ${\cal P}$ is a point process
where each atom represents a potential infection event. Define
$\abs{\mathcal{P}} \coloneqq \int \d {\cal P}(a)$ which is interpreted as the total
number of potential infections (or {\it contacts}) along the course of
infection. We define a triplet of random variables $(W, \mathcal{P}^{\star},
X^{\star}) \equiv (W, \mathcal{X}^{\star})$ valued in 
$\RR_+ \times {\cal M} \times {\cal S} \equiv \RR_+ \times \mathscr{X}$ 
such that for every bounded continuous function $f$
\begin{eqnarray*}
    \EE\bigg( f(W, \mathcal{P}^{\star}, X^{\star}) \bigg) & = & \frac{1}{R_0} \EE\bigg(  \int f(a, {\cal P}, X) \d {\cal P}(a)  \bigg) \\
    & = & \frac{1}{R_0} \EE\bigg(  |{\cal P}| \  \times  \int  \frac{1}{|{\cal P}|} f(a, {\cal P}, X) \d {\cal P}(a)  \bigg)
\end{eqnarray*}
In words, we first bias the pair ${\cal X} \ = ({\cal P}, X)$ by $\abs{\cal P}$. 
Conditional on the resulting biased pair $\mathcal{X}^\star = (\mathcal{P}^\star, 
X^\star)$, the r.v.\ $W$ is obtained by picking an atom of the infection
measure $\mathcal{P}^\star$ uniformly at random. 

\begin{definition}[Campbell and Palm measures]
\label{def:campbel}
The law of $(W, \mathcal{X}^\star)$ is the Campbell's measure associated to
${\cal X}$ \cite{baccelli2020random}. The Palm measure at $a\in\RR_+^*$
is defined as the distribution of the random pair ${\cal X}^\star$
conditioned on the event $\{W=a\}$. We will use the notation
$\mathcal{X}^{(a)}$ for a random variable with the Palm measure at $a$.
See again \cite{baccelli2020random} for a precise definition of this
conditioning.
\end{definition}

Recall that $\tau$ is the intensity measure of $\mathcal{P}$ defined in
\eqref{eq:intensity}, and that we can write it as $\tau = R_0 \nu$,
where the total mass $R_0$ and the probability measure $\nu$ are defined
in \eqref{eq:R0}. The next result is standard from Palm measure theory.

\begin{lemma}
The random variable $W$ is distributed according to $\nu$.
\end{lemma}

\subsection{Definition of the Poisson tree}
\label{def:poisson-tree}

Recall that we have defined $\bar{\tau}$ in \eqref{eq:tauBar} by
\[
    \bar{\tau}(u) = \int_0^\infty g(a)\tau(a+u)\,\d a, \qquad u\geq 0,
\]
where $g$ is the initial age density of infected individuals, and that we
write $\bar{\tau} = \bar{R}_0 \bar{\nu}$ where $\bar{R}_0$ is the mass of
$\bar{\tau}$ and $\bar{\nu}$ the renormalized probability measure, see
\eqref{eq:R0bar}. Let us now consider a pair of random variables
$(\bar{W},Z)\in \RR_+^2$ with joint density 
\begin{eqnarray} \label{def:g}
    \forall w,z>0, \ \ G(w,z)  & = & \frac{1}{\bar R_0} g(z) \tau(w+z).
\end{eqnarray}
In particular, the first coordinate is distributed according to $\bar \nu$.

We now construct a Poisson marked random tree $\mathcal{H}$ in
two consecutive steps. (This extends the construction of \Cref{SS:dual}
to the case $c \not \equiv 1$.) First, the graph structure  of
$\mathcal{H}$ depends on the two positive real parameters $S_0R_0$,
$I_0\bar R_0$, and second the random edge lengths and the marks are
assigned through two probability distributions $\nu$, $\bar{\nu}$ and the
Palm measures described in the previous section. 

\medskip
\noindent
{\bf Step 1. Graph structure.}
The graph structure is given by a Poisson Galton--Watson tree with two types:
\begin{itemize}
    \item Start from a root $\emptyset$ of type $(S)$.
    \item Susceptible $(S)$ nodes have independent
        $\mathrm{Poisson}(S_0R_0)$ susceptible $(S)$ offspring, and
        $\mathrm{Poisson}(I_0\bar R_0)$ infected $(I)$ offspring.
    \item $(I)$ nodes have no offspring.
\end{itemize}
In the following, let us consider the edges of the tree as being {\it oriented towards the root}. 

\medskip
\noindent
{\bf Step 2. Decoration.}
Given the tree structure with distinguished $(S)$ and $(I)$ vertices, we
now assign a marking $m_i=(Z_i, {\cal X}_i)$ to every vertex $i$, and a
length $a_e$ to every edge $e$ as follows. If $i=\emptyset$, then
$m_\emptyset = (0, {\cal X}_{\emptyset})$ where ${\cal X}_\emptyset$ is
distributed as ${\cal X}$. For every $i\neq \emptyset$, there exists a
unique oriented edge $e=(i,j)$ originated from $i$ and 
\begin{itemize}
 \item If $i\in (I)$, let $(a_e, Z_i)$ be chosen according to the density
     $G$. If $i\in (S)$, then $Z_i=0$ and $a_e$ is chosen according to
     $\nu$.
 \item Conditional on $(a_e, Z_i)$, the variable $\mathcal{X}_i$
     has the Palm measure $\mathcal{X}^{(a_e+Z_i)}$ evaluated at $a_e+Z_i$. 
\end{itemize}

\begin{remark} ~ \vspace{-1.2ex}
\begin{itemize}
\item If $e=(i,j)$ with $i\in(S)$ then $(a_e, {\cal X}_i)$ has 
the Campbell measure introduced in Definition~\ref{def:campbel}.
\item If $i\in(I)$, then $a_e$ is distributed according to $\bar \nu$.
\end{itemize}
\end{remark}

The random tree $\mathcal{H}$ will correspond to the local limit of the pogm graph
$\mathcal{G}^N(x)$ conditioned on $\{Z_x=0\}$. 
Let us now consider the infection process on ${\cal H}$ introduced in
Section~\ref{sec:def:active-goedesics}. Conditional on ${\cal H}$,
we endow each oriented edge $e$ with a uniform random variable $s_e$ (the
intensity of the contact). As pointed out in Definition~\ref{def:active-goedesics},
those r.v's allow to determine whether a path is active or not and to determine the active geodesic at the root.

Define $\sigma^\infty$
as the length of the active geodesic in ${\cal H}$ from the set of $(I)$ leaves to the root $\emptyset$. 
The following
key result connects the distribution of $\sigma^\infty$ to the delay equation.

\begin{proposition} \label{prop:dual}
    Define
    \[
        \forall t \ge 0,\quad B(t) \coloneqq S_0 \mathbb{P}(\sigma^{\infty} \le t).
    \]
    Then $B$ solves the delay equation~\eqref{eq:delay}.
\end{proposition}

\begin{proof} 
    As we have assumed that $\tau$ has a density w.r.t.\ the Lebesgue
    measure, it is clear that this also holds for the distribution of
    $\sigma^\infty$. We denote its density by $f$. Let $K$, resp.\
    $\bar{K}$, be the number of type ($S$), resp.\ type ($I$), children
    of the root of $\mathcal{H}$. Let $(\mathcal{H}_1,\dots,
    \mathcal{H}_K)$ denote the subtrees attached to the root $\emptyset$
    which are growing out of the type ($S$) children of the root. Let
    $(\sigma^\infty_1, \dots, \sigma^\infty_K)$ be the corresponding
    infection times, $\sigma^\infty_i$ being obtained by determining the
    length of the active geodesic from the vertices of type ($I$) to the
    root in the tree $\mathcal{H}_i$. Moreover, let $(W_1, \dots, W_K)$ and
    $(\bar{W}_1, \dots, \bar{W}_{\bar{K}})$ be the lengths of the edges
    ending at $\emptyset$ and starting from an ($S$) and an ($I$) children
    respectively. (Recall that the edges of the Poisson tree are directed
    towards the root.) Finally, with a slight abuse of notation, let
    $s_i$ be the contact intensity on the edge with length $W_i$. Let $\bar
    s_i$ be defined analogously. Define
    \[
        \chi_i \coloneqq \indic{s_i \leq c(W_i + \sigma^{\infty}_i)},
        \quad \bar{\chi}_i = \indic{\bar s_i \leq c(\bar{W}_i)}.
    \]
    By definition of the active geodesic, we have that 
    \begin{equation} \label{eq:recursiveGeodesic}
        \sigma^\infty =
        \Big(\min_{1 \le i \le K} \big\{ \chi_i (W_i+\sigma^\infty_i) 
        + (1-\chi_i) \times \infty  \big\} \Big) \wedge 
        \Big( \min_{1 \le i \le \bar{K}} \big\{ \bar{\chi}_i\bar{W}_i +
        (1-\bar{\chi}_i) \times \infty \big\} \Big),
    \end{equation}
    with the convention $0 \times \infty = 0$. Define $G(t) =
    \mathbb{P}(\sigma^\infty > t)$. Let $W$ and $\bar W$ be distributed
    according to $\nu$ and $\bar \nu$ respectively. By the branching
    property, conditional on $K$ and $\bar{K}$, $(\sigma^\infty_i,
    W_i, \chi_i)_i$ and $(\bar{W}_i, \bar{\chi}_i)_i$ are two independent
    collections of i.i.d.\ random variables. Moreover 
    \begin{align}
    \begin{split} \label{eq:probaGeodesic}
        \P\big( \chi_i (W_i+\sigma^\infty_i) 
        + (1-\chi_i) \times \infty \le t \big)
        &=
        \EE\big(  c(W+\sigma^\infty) \indic{\sigma^\infty+W \le t}
        \big) \\
        \P\big( \bar{\chi}_i \bar{W}_i 
        + (1-\bar{\chi}_i) \times \infty \le t \big)
        &=
        \EE\big(  c(\bar{W}) \indic{\bar{W} \le t}
        \big)
    \end{split}
    \end{align}
    Using these expressions, \eqref{eq:recursiveGeodesic} and the
    branching property we have
    \begin{align*}
        G(t) &= \EE \Big\{ \Big(1 - \EE\big( c(\sigma^\infty + W)
            \indic{\sigma^\infty+W \le t} \big) \Big)^K
            \Big(1 - \EE\big( c(\bar{W}) \indic{\bar{W} \le t} \big)
            \Big)^{\bar{K}} \Big\} \\
            &= \EE \Big\{ \Big(1 - \int_0^t \int_0^{t-a} c(a+s) f(s)
            \,\d s \,\nu(\d a) \Big)^K
            \Big(1 - \int_0^t c(s) \,\bar{\nu}(\d s)\Big)^{\bar{K}} \Big\} \\
            &= \begin{aligned}[t]
                \exp\Big(-S_0\int_0^t \int_0^{t-a} c(a+s)&f(s)\tau(a) \,\d s\,\d a\\
                 &- I_0\int_0^t g(a)\int_a^\infty c(u-a) \tau(u) \,\d u \,\d a\Big),
                \end{aligned}
    \end{align*}
    where, in the last equality, we have used the generating function of a
    Poisson distribution. It now follows that $B(t) = S_0(1-G(t))$
    satisfies~\eqref{eq:delay}. 
\end{proof}

\subsection{The infection path conditioned on its length}
\label{sect:spinal-process}

Let us consider the infection process on ${\cal H}$
as described in the previous section.
For every realization in $\{\sigma^\infty<\infty\}$, define ${\cal
R}^\infty$ to be the infection path from $\emptyset$ to  the $(I)$ leaves
in ${\cal H}$, and let ${\cal A}^\infty$ be the ancestral process defined
analogously to Definition~\ref{def:infection-path}.
In this section, we ask the following question: {\it conditional on the active geodesic 
to be of length $t$,
what is the distribution
of the vector of infection times along the geodesic?} In order to give an answer to this question, we start with some definition. 

Let us consider ${\cal R}^\infty$ to be the infection path from
$\emptyset$ to the ($I$) leaves in ${\cal H}$ --- see Definition
\ref{def:infection-path}. Our aim is to provide a description of ${\cal
R}^\infty$ conditional on $\{\sigma^\infty=t\}$. Define the process $\hat
R^{(t)} \equiv \hat R$ as the $\RR$-valued, nonincreasing Markov chain,
started from $t$ and stopped upon reaching $(-\infty,0]$, with transition
kernel $Q(x,y)$ defined for all $x>0$ by
\begin{align*}
\forall y \geq x, \quad  Q(x,y) & \coloneqq \; 0\\
\forall y < x, \quad  Q(x,y) & \coloneqq \; \frac{S(x) c(x) b(y)}{b(x)}  \tau(x-y),
\end{align*}
where $b$ is extended to the negative half-line with $b(-t) \coloneqq
I_0g(t)$. The fact that $Q$ defines a transition kernel follows from the
renewal equation for $b$, which is obtained by differentiating
\eqref{eq:delay} with respect to $t$:
\begin{equation}\label{ren-b}
\forall t\geq 0, \qquad b(t) \ = \ c(t) S(t) \int_{-\infty}^t b(a) \tau(t-a) da.
\end{equation}
Define 
\[
    \hat L^{(t)} \ \coloneqq \ \hat L \ = \ \inf\{k : \hat R^{(t)}_k \leq 0\}.
\]
In the next proposition, we slightly abuse notation and identify $\hat R^{(t)}$
with its finite-length restriction to $[\hat L]$.

\begin{proposition} \label{prop:conditionalPath}
Let ${\cal R}^\infty$ be the infection path from $\emptyset$ to the $(I)$ leaves.
Conditional on $\{{\cal R}^\infty(0) = \sigma^\infty=t\}$, 
\[
    \mathcal{R}^\infty \ = \ \hat R^{(t)} \quad \text{in law}.
\]
\end{proposition}

\begin{proof}
Recall that $\sigma^\infty = {\cal R}^\infty(0)$ is a random variable
valued in $\RR_+\cup \{\infty\}$. By \Cref{prop:dual}, the density of the
random variable $\sigma^\infty$ on $\RR_+$ is given by $S_0^{-1} b(t)$.
Let $F$ be the joint probability density of the random pair 
$({\cal R}^\infty(1), {\cal R}^\infty(0)-  {\cal R}^\infty(1))$, and
define
\[
    \forall t \text{ and } x \le t, \quad F^{(t)}(t-x) 
    \ \coloneqq \ 
    \frac{F(x,t-x)}{S_0^{-1} b(t)},
\]
so that $F^{(t)}$ corresponds to the density of the increment ${\cal
R}^\infty(0)-  {\cal R}^\infty(1)$ conditioned on $\{{\cal
R}^\infty(0)=t\}$. Since ${\cal H}$ is a Poisson random tree it is
sufficient to understand the first step of the infection path, i.e., we
need to show that
\begin{equation}\label{eq:f}
F^{(t)}(t-x) \ = \ \frac{c(t) S(t) b(x) \tau(t-x)}{b(t)}.
\end{equation}
We use the same notation as in the proof of \Cref{prop:dual} and we
distinguish between two cases.

\bigskip

{\bf Case 1: $x\in[0,t]$.} In this case, the first individual along the
geodesic is of type ($S$). Let us work conditional on $(K,\bar{K})$ and 
compute the density $F(x,t-x)$. Fix a child $i \le K$ of type ($S$) of
the root $\emptyset$. By construction of the tree $\mathcal{H}$, the
active geodesic leading to $i$ and the length of the edge $e_i$ going out
of $i$ toward the root are independent. Their joint density at $(x, t-x)$
is $S_0^{-1}b(x)\nu(t-x)$ by \Cref{prop:dual}. For individual $i$ to be
part of the active geodesic leading to $\emptyset$, the edge $e_i$ needs
to be active, which occurs with probability $c(t)$, and the shortest
active path going through any of the other children of the root must be
longer than $t$. Using the expression \eqref{eq:probaGeodesic}, the
probability of the latter event is 
\[
    \big(\EE\big(  c(W+\sigma^\infty) \indic{\sigma^\infty+W \le t}
    \big)\big)^{K-1}
    \times
    \big(
    \EE\big(  c(\bar{W}) \indic{\bar{W} \le t}
    \big)\big)^{\bar{K}}.
\]
Summing over all $K$ children of type ($S$) yields that
\begin{align*}
F^{(t)}(t-x) = & \; \frac{c(t)  S_0^{-1} b(x) \nu(t-x) }{ S_0^{-1}b(t)} \\
    &\times \EE\Big\{ 
        \indic{K\geq1} 
        K\big(1 - \EE\big( c(W+\sigma^\infty) \indic{W +\sigma^\infty \le t} \big) \big)^{K-1} 
        \big(1 - \EE\big( c(\bar{W}) \indic{\bar{W} \le t} \big)\big)^{\bar{K}}
        \Big\} \\
    = & \; \frac{c(t)b(x) \nu(t-x)}{b(t)} \times  S_0 R_0 \\ 
    &\times \EE\Big\{
        \big(1 - \EE\big( c(W+\sigma^\infty) \indic{W +\sigma^\infty \le t} \big) \big)^{K}
        \big(1 - \EE\big( c(\bar{W}) \indic{\bar{W} \le t} \big)\big)^{\bar{K}} 
    \Big\} \\
    = & \; \frac{c(t)b(x) \tau(t-x)}{b(t)} \\
    &\times S_0 \EE\Big\{
        \big(1 - \EE\big( c(W+\sigma^\infty) \indic{W +\sigma^\infty \le t} \big) \big)^{K} 
        \big(1 - \EE\big( c(\bar{W}) \indic{\bar{W} \le t} \big)\big)^{\bar{K}} 
    \Big\}
\end{align*}
where in the second line, we used the fact that $K$ is Poisson($S_0 R_0$) (so that the size-biased version of $K$ is identical in law to $K+1$). In the third line, we used the relation 
$\tau(u) = R_0 \nu(u)$. In \Cref{prop:dual}, we showed that 
\[
S_0 \EE\Big\{\Big(1 - \EE\big( c(V+\sigma^\infty) \indic{V +\sigma^\infty \le t} \big) \Big)^{K} \Big(1 - \EE\big( c(\bar{V}) \indic{\bar{V} \le t} \big)\Big)^{\bar{K}} \Big\} 
\ = \ 
S_0 - B(t) = S(t).
\]
This shows \eqref{eq:f}.

\bigskip

{\bf Case 2: $x\leq 0$.} On this event the first vertex along the
transmission chain is of type ($I$). We use the same argument as in the
case $x > 0$. Let $i \le \bar{K}$ be a child of $\emptyset$ of type
($I$). Again, for this individual to be in the active geodesic, all paths
going from an ($I$) individual to the root and not going through $i$ need
to be longer than $t$, and the edge from $i$ to $\emptyset$ needs to be
active. In this case, $\mathcal{R}^\infty(1) = -Z_i$ and
$\mathcal{R}^{\infty}(0) = \bar{W}_i$, where $(\bar{W}_i, Z_i)$ have
joint density $G$ defined in \eqref{def:g}. Thus the density of
$(\mathcal{R}^\infty(1), \mathcal{R}^\infty(0)-\mathcal{R}^\infty(1))$ at 
$(x, t-x)$ is $G(t, -x) = g(-x) \tau(t-x) / \bar{R}_0$.
This together with \eqref{eq:probaGeodesic} lead to
\begin{align*}
F^{(t)}(t-x) 
    = & \;  \frac{c(t)g(-x) \tau(t-x) /\bar R_0 }{S_0^{-1} b(t)} \\
      &\times \EE \Big\{
          \big(1 - \EE\big( c(W+\sigma^\infty) \indic{W +\sigma^\infty \le t} \big) \big)^{K}
          \indic{\bar K\geq1} \bar K 
          \big(1 - \EE\big( c(\bar{W}) \indic{\bar{W} \le t} \big)\big)^{\bar{K}-1} 
      \Big\} \\
    = & \; \frac{c(t)b(x) \tau(t-x) /\bar R_0}{I_0 S_0^{-1}  b(t)} \times I_0 \bar R_0 \\ 
      &\times \EE\Big\{
          \big(1 - \EE\big( c(W+\sigma^\infty) \indic{W +\sigma^\infty \le t} \big) \big)^{K} 
          \big(1 - \EE\big( c(\bar{W}) \indic{\bar{W} \le t} \big)\big)^{\bar{K}} 
      \Big\} \\
    = & \; \frac{c(t)b(x) \tau(t-x)}{b(t)}  \\ 
      &\times S_0 \EE\Big\{
         \big(1 - \EE\big( c(W+\sigma^\infty) \indic{W +\sigma^\infty \le t} \big) \big)^{K} 
         \big(1 - \EE\big( c(\bar{W}) \indic{\bar{W} \le t} \big)\big)^{\bar{K}} 
     \Big\} \\
    = & \; \frac{c(t) S(t) b(x) \tau(t-x)}{b(t)}. 
\end{align*}

\end{proof}

\subsection{Harmonic transform}
\label{sect:h-transform}

In this section, we prove that the path $\hat R^{(t)}$
is the $h$-transform of a renewal process stopped upon reaching $(-\infty,0]$.
Throughout this section,
we assume the existence of a unique Malthusian parameter $\alpha\in\mathbb{R}$ such that
$$
\int \exp(-\alpha a)\tau(a) \,\d a \ = \ 1.
$$
We define the probability density on ${\mathbb R}_+^*$
$$
\forall a>0, \ \  r(a) \ \coloneqq \ \exp(-\alpha a) \tau(a).
$$
Let $(Y_i)$ be a sequence of i.i.d.\ random variables with probability density $r$.
Let $t>0$ and define the renewal process $R^{(t)} \equiv R$ as follows
$$
\forall k\geq1, \ \ R^{(t)}_k \ = \ t- \sum_{i=1}^k Y_i, \ \ R^{(t)}_0 = t.
$$ 
We couple the renewal process $R$ with a random variable $K^{(t)}\equiv K$ valued in $\NN\cup \{\infty\}$
such that
conditional on $R$, 
$$
\forall j \geq 0, \  \P(K = j \mid R) = 
\ell(R_0) \cdots \ell(R_{j-1}) \bigg(1-\ell(R_{j})\bigg),
$$
with $\ell(x) = \indic{x > 0} S(x)c(x) + \indic{x \leq 0}$.

\begin{remark}
Recall that $c$ and $S$ are valued in $[0,1]$.
Think of $K$ as a killing time for the process $R$, i.e., at site $x>0$,
$R$ dies with probability $1-S(x)c(x)$, or makes a transition according to the distribution $r$
with the remaining probability.
Since by definition $\ell(x) = 1$ for all $x\leq 0$, if $R$ reaches a negative state without being killed, it can no longer be killed.
\end{remark}

Consider the filtration $({\cal F}_k; k\geq0)$ where 
$$
{\cal F}_k \ = \ \sigma( (R_0, \chi_0),\cdots, (R_k, \chi_{k})  ), \ \ \mbox{where $\chi_k \ = \ \indic{K \geq k}$},
$$
and define the reaching time of $(-\infty,0]$ as $L \coloneqq \inf\{k : R_k \leq 0\}$.

\begin{lemma}
Define 
$$
M_k \coloneqq b(R_{k\wedge L}) e^{-\alpha R_{k\wedge L}} \chi_{k}.
$$ 
The process $(M_{k}; k\geq0)$ 
is a martingale with respect to the filtration $({\cal F}_k; k\geq0)$.
\end{lemma}
\begin{proof}
Let us compute the conditional expectation $\EE( M_{k+1} \  | \ {\cal F}_k )$ for a realization on the event $A_k\coloneqq\{R_k>0, K \geq k\}$. 
The martingale property is obviously satisfied for any realization on the complementary event.
Using the renewal equation (\ref{ren-b}) for $b$, we have
\begin{align*}
\mathbbm{1}_{A_k}\EE( M_{k+1} \  | \ {\cal F}_k ) & = \mathbbm{1}_{A_k} \EE\bigg( b(R_{k+1}) e^{-\alpha R_{k+1}} \indic{K \geq k+1} \mid {\cal F}_k  \bigg)  \\
& = \mathbbm{1}_{A_k} S(R_k) c(R_k) \int_0^\infty b( R_{k} - a ) e^{-\alpha (R_k-a)} \tau(a) e^{-\alpha a} \,\d a \\
& = \mathbbm{1}_{A_k} e^{-\alpha R_k} S(R_k) c(R_k) \int_0^\infty b( R_{k} - a )  \tau(a)\,\d a  \\
& = \mathbbm{1}_{A_k} b(R_k) e^{-\alpha R_k}. \qedhere
\end{align*}
\end{proof}

\begin{proposition}
Let $h(s,u) \coloneqq b(s)e^{-\alpha s} u$  and consider the $h$-transform of 
the two dimensional process $(R,\chi)$.
Then the process $\hat R$ is the first coordinate of the $h$-transformed process.
\end{proposition}
\begin{proof}
On the one hand, the previous lemma implies that  $h$ is a harmonic function 
for the bivariate process $(R, \chi)$.
On the other hand, the transition kernel $\hat Q$ for the  $h$-transformed process  
can be rewritten explicitly as 
\begin{align*}
\forall x, y ; \   \hat Q\bigg((x,1),(y, 0)\bigg) & \coloneqq  0 \\
\forall x \leq  y; \forall \epsilon\in\{0,1\}\ , \   \hat Q\bigg((x,1),(y,\epsilon)\bigg) & \coloneqq  0\\
\forall x > y; \  \hat Q\bigg((x,1),(y,1)\bigg) & \coloneqq  \frac{b(y) e^{-\alpha y }}{b(x)e^{ -\alpha x}}  \ S(x) c(x) \;  \tau(x-y) e^{-\alpha(x-y)}  \\
& =  \frac{b(y) S(x) c(x)}{b(x)}  \;  \tau(x-y) 
\end{align*}
It is now straightforward to check that $\hat R$ is identical in law with the first coordinate of the $h$-transformed process.
\end{proof}

Let $P$ be the law of the bivariate path $(R, \chi)$ stopped at $L=\inf\{k : R_k \leq 0 \}$.
Let $\hat P$ be the law of $h$-transform $(\hat R, \hat \chi)$ stopped at $\hat L = \ \inf\{k : \hat R_k \leq 0 \}$.
Then $\hat P \ll P$ and the Radon--Nikodym derivative is given by
$$
\frac{\d \hat P}{\d P} \ = \ \frac{ b(R_L) \exp(-\alpha R_L) }{b(t) \exp(-\alpha t)}  \chi_L.
$$

This immediately entails the following result.

\begin{proposition}\label{prop:exx}
Assume that $g(t)= \alpha \exp(-\alpha t)$. Then $\hat P$
is obtained by conditioning the renewal process $R$ on not being killed before time $L$, and $b(t)$ can be written:
\[
  b(t) = \alpha e^{\alpha t} P(R^{(t)}\text{ is not killed before time }L).
\]
\end{proposition}

\begin{remark}
Consider the linearized version of the Kermack--McKendrick equation
    \begin{align*} 
        \begin{split}
        \partial_t n(t,a) + \partial_a n(t,a) &= 0 \\
        \forall t \ge 0,\; n(t, 0) &= c(t)\int_0^\infty n(t,a) \,\tau(\d
        a) \\
        \forall a \ge 0,\; n(0, a) &= I_0 g(a)\\
        \end{split}
    \end{align*}
    obtained from \eqref{eq:mckvf} by assuming $S(t)=1$. This can be thought as  
    the age structure of a population where susceptibles are in excess.
    One can check that if $g(a)=\alpha \exp(-\alpha t)$, then $b_{\mathrm{lin}}(t)
    \coloneqq n(0,t) =  \alpha e^{\alpha t}$. As a consequence, Proposition
    \ref{prop:exx} can be rewritten as 
\[
    b(t) \ = \ b_{\mathrm{lin}}(t) \  P(R^{(t)}\text{ is not killed before time }L).
\]
\end{remark}

We close this section by a brief discussion on the previous result. In
\cite{foutelrodier2020individualbased}, we considered a ``linearized''
version of the present model by making the simplifying assumption that
susceptible individuals are always in excess (branching assumption), so
that the epidemic is described by a Crump--Mode--Jagers process. When
$c\equiv1$ and $R_0>1$, the process is supercritical. Starting from a
single infected individual, there is a positive probability of
non-extinction and conditional on this event, the number of infected
grows exponentially at rate $\alpha>0$. Further, it is well known from
the seminal work of Jagers and Nerman  \cite{Nerman1984} that  under mild
assumptions,
\begin{enumerate}
\item the age structure 
of the population converges to the exponential profile $g(t) = \alpha \exp(-\alpha t)$ mentioned in \Cref{prop:exx}.
\item the infection path --- interpreted as the ancestral line in the work of Jagers and Nerman --- is well described by the renewal process $R$. 
More precisely, if one sample an infected individual at a large time $t$, its infection path converges to the renewal process $R$. 
\end{enumerate}
We can draw two conclusions out of those observations.  As a consequence
of the first item,  the age structure  $g(t)= \alpha \exp(-\alpha t)$
could be interpreted as the age structure emerging from a single infected
individual in the past (provided that the initial fraction of infected
individuals in our model is small). The second conclusion is that the
effect of the conditioning in \Cref{prop:exx} encodes the effect of the
saturation and the contact rate $c$ on the genealogy. Recall that in the
absence of saturation (branching approximation) and full contact rate
($c\equiv1$), the infection path is distributed as the renewal process.
When those effects are taken into account, \Cref{prop:exx} indicates that
the law of the infection path is twisted in such a way that infection
paths with infection occurring at low susceptible frequency (i.e.\ low
values of $S$) and high contact rates $c$ are favored. This is consistent
with the intuition that ancestral infections tend to be biased towards
periods when many infections occurred.

\section{Convergence of the infection graph} 
\label{SS:proofLLN}

We show in this section that the Poisson random tree $\mathcal{H}$
constructed previously corresponds to the local weak limit of ($S$)
vertices in the infection graph $\mathcal{G}^N$. This entails that 
the empirical distribution of any continuous functional of the graph 
in the local topology converges to the law of the corresponding
functional for $\mathcal{H}$. In particular we will deduce our two mains
results, the convergence of the age structure and that of the historical
process, by viewing the age of an individual $x$ and its transmission chain
as functionals of the active geodesic in $\mathcal{G}^N$ leading to $x$.
The key result of this section is the following.

\begin{proposition} \label{prop:LWconvergence}
    The sequence of infection graphs $(\mathcal{G}^N)_N$ converges in
    probability in the local weak sense to a random pogm tree
    $\mathcal{T}$ such that
    \begin{itemize}
        \item with probability $I_0$, $\mathcal{T}$ is made of a
            single $(I)$ vertex $\emptyset$, whose mark $(Z_\emptyset, 
            \mathcal{X}_\emptyset)$ is distributed as $Z_\emptyset \sim g(a)
            \diff a$ and $\mathcal{X}_\emptyset \sim (\mathcal{P}, X)$;
        \item with probability $S_0$, $\mathcal{T}$ is distributed as the
            random tree $\mathcal{H}$ of Section~\ref{def:poisson-tree}.
    \end{itemize}
\end{proposition}

In other words, the tree $\mathcal{H}$ constructed in
Section~\ref{def:poisson-tree} corresponds to the law of $\mathcal{T}$,
conditioned on starting from an ($S$) vertex. 

\begin{lemma} \label{lem:convMeasure}
    For each $N$, let $(X^N_i;\, i \le N)$ be some exchangeable random
    variables in some Polish state space, and $(X_1, X_2)$ be two
    independent random variables with distribution $\mathcal{L}(X)$. Then
    \[
        (X^N_1, X^N_2) \to (X_1, X_2) \iff \frac{1}{N} \sum_{i=1}^N
            \delta_{X^N_i} \to \mathcal{L}(X)
    \]
    where the two convergence are in distribution.
\end{lemma}

\begin{proof}
    By exchangeability, for any continuous bounded $\phi, \psi$, 
    \begin{equation} \label{eq:double-cv-exchangeability-lemma}
            \EE\Big[ \Big(\frac{1}{N} \sum_{i=1}^N \phi(X_i^N)\Big) 
                 \Big(\frac{1}{N} \sum_{j=1}^N \psi(X_j^N)\Big) \Big]
                 = \EE\big[ \phi(X_1^N)\psi(X_2^N) \big] + O\big(\tfrac{1}{N}\big).
    \end{equation}
    If the random measure converges, then by dominated convergence
    \[
        \EE\Big[ \Big(\frac{1}{N} \sum_{i=1}^N \phi(X_i^N)\Big)
                 \Big(\frac{1}{N} \sum_{j=1}^N \psi(X_j^N)\Big) \Big]
        \longrightarrow 
        \EE[\phi(X_1)]\EE[\psi(X_2)]
    \]
    showing that the pair $(X^N_1, X^N_2)$ converges in distribution to
    $(X_1, X_2)$. Conversely, using again \eqref{eq:double-cv-exchangeability-lemma},
    the convergence of $(X^N_1, X^N_2)$ entails that 
    \[
        \EE\Big[ \frac{1}{N} \sum_{i=1}^N \phi(X_i^N) \Big] =
        \EE[ \phi(X^N_1) ] \to \EE[\phi(X_1)],\quad
        \EE\Big[ \Big(\frac{1}{N} \sum_{i=1}^N \phi(X_i^N)\Big)^2 \Big]
        \longrightarrow
        \EE[\phi(X_1)]^2.
    \]
    These two estimates prove that $\frac{1}{N} \sum_{i=1}^N \phi(X_i^N)$
    converges in distribution to $\EE[\phi(X_1)]$, which in turn shows
    that the measure $\frac{1}{N} \sum_{i=1}^N \delta_{X^N_i}$ converges
    to $\mathcal{L}(X)$, see for instance \cite[Theorem 4.11]{Kal17}.
\end{proof}

\begin{proof}[Proof of Proposition~\ref{prop:LWconvergence}]
    We prove the result in three steps. First, we show the local weak
    convergence of the graph structure (without the marking) towards a
    limiting Galton--Watson tree $\widetilde{\mathcal{T}}$. We make use of known
    results on the local weak convergence of configuration models. Then we
    show that $(\mathcal{G}^N)_N$ (with the marking) converges to the tree
    $\mathcal{T}$ obtained by marking $\widetilde{\mathcal{T}}$
    appropriately. Finally we prove that the law of the limiting tree
    $\widetilde{\mathcal{T}}$, conditional on starting from an ($S$)
    vertex and after removing all edges pointing towards an ($I$) vertex,
    is distributed as the Poisson tree $\mathcal{H}$.

    \medskip
    \noindent
    \textbf{Step 1.} Recall that the infection graph $\mathcal{G}^N$ can be constructed
    as a directed configuration model, see the notation in
    Proposition~\ref{prop:configurationModel}. We will use the known fact
    that the local weak limit of a configuration model is a
    Galton--Watson tree \cite[Section~2.2.2]{van2017stochastic}. We make
    use of a version of this result for directed graphs derived in
    \cite[Proposition~6.2]{garavaglia20}.

    The local weak convergence in \cite{garavaglia20} is derived for a
    different class of oriented graphs than the pogm graphs introduced in
    this work. Namely, edges have no lengths and the vertices are marked
    with their out-degrees. Accordingly, let us denote by
    $\tilde{\mathcal{G}}^N$ the oriented marked graph obtained by
    replacing the marks $(m_x)_x = (Z_x, \mathcal{X}_x)_x$ by the mark
    $(\tilde{m}_x)_x = (\abs{\widehat{\mathcal{P}}_x})_x$ and removing
    the edge lengths. Recall the notation $(D^\Out_x)_x$ and
    $(D^\In_x)_x$ for the collection of in and out degree in
    $\tilde{\mathcal{G}}^N$. Three conditions need to be checked on this
    degree sequence to obtain the local weak convergence of
    $\tilde{\mathcal{G}}^N$, see Condition~6.1 of \cite{garavaglia20},
    \begin{enumerate}
        \item[(a)] for any positive bounded function $\phi$, in probability,
            \[
                \frac{1}{N} \sum_{x \in [N]} \phi(D^\Out_x, D^\In_x)
                \longrightarrow
                \EE[\phi(\mathcal{D}^\Out, \mathcal{D}^\In)];
            \]
        \item[(b)] we have
            \[
                \frac{1}{N} \sum_{x \in [N]} D^\Out_x 
                \longrightarrow
                \EE[\mathcal{D}^\Out],\qquad
                \frac{1}{N} \sum_{x \in [N]} D^\In_x 
                \longrightarrow
                \EE[\mathcal{D}^\In]
            \]
            in probability, and $\EE[\mathcal{D}^\Out] = \EE[\mathcal{D}^\In]$; and
        \item[(c)] for any positive bounded function $\phi$, if $L^N =
            D^\Out_1+\dots+D^\Out_N$, 
            \[
                \frac{1}{L^N} \sum_{x \in [N]} D^\Out_x \phi(D^\Out_x, D^\In_x)
                \longrightarrow
                \EE[\phi(\mathcal{D}^{\star\Out}, \mathcal{D}^{\star\In})]
            \]
            in probability (note that we have removed a $1/N$ factor
            compared to \cite{garavaglia20} that should not appear);
    \end{enumerate}
    for some random pair $(\mathcal{D}^\Out, \mathcal{D}^\In)$, and where
    $(\mathcal{D}^{\star\Out}, \mathcal{D}^{\star\In})$ is obtained by
    size-biasing $(\mathcal{D}^\Out, \mathcal{D}^\In)$ by its first
    coordinate.

    We check condition (a) by computing the second moment of
    the empirical distribution of degrees. Since the in-degrees follow
    the multinomial distribution \eqref{eq:multinomial}, we have that
    \[
        (D^\In_1, D^\In_2) \sim 
        \mathrm{Multinomial}\big(D^\Out_1 + D^\Out_2 + \sum_{i=3}^N D^\Out_i; (\tfrac{1}{N}, \tfrac{1}{N}) \big)
    \]
    so that, conditional on $(D^\Out_1, D^\Out_2)$, 
    \[
        (D^\In_1, D^\In_2) 
        \longrightarrow
        (\mathcal{D}^\In_1, \mathcal{D}^\In_2)
    \]
    which are independent Poisson random variables with mean
    $\EE[\abs{\widehat{\mathcal{P}}}] = S_0 R_0 + I_0 \bar{R}_0$. Using
    Lemma~\ref{lem:convMeasure} proves that (a) holds where
    $(\mathcal{D}^\Out, \mathcal{D}^\In)$ are independent r.v.\ with
    $\mathcal{D}^\Out \sim \abs{\widehat{\mathcal{P}}}$ and
    $\mathcal{D}^\In \sim \mathrm{Poisson}(S_0 R_0 + I_0 \bar{R}_0)$.

    Point (b) is a direct application of the law of large numbers
    \[
        \frac{1}{N} \sum_{i=1}^N D^\Out_i \longrightarrow \EE[ \mathcal{D}^\Out ],
    \]
    and point (c) follows from (a) and (b): using point (a) with
    $\phi(d,d') = \indic{(d,d') = (k,k')}$ we have that
    \[
        \frac{1}{N} \Card \{ x : (D^\Out_x, D^\In_x) = (k,k')\} \longrightarrow
        \P( \mathcal{D}^\Out = k, \mathcal{D}^\In = k'),
    \]
    in probability, and combining this with point (b) we have
    \begin{equation} \label{eq:sizeBiasedConv}
        \frac{k}{L_N} \Card \{ x : (D^\Out_x, D^\In_x) = (k,k')\} \longrightarrow
        \frac{k}{\EE[\mathcal{D}^\Out]} \P( \mathcal{D}^\Out = k, \mathcal{D}^\In = k'),
    \end{equation}
    in probability. This shows (c) for our specific choice of $\phi$. 
    For a general $\phi$, up to extracting a subsequence, let us assume
    that a.s.\ \eqref{eq:sizeBiasedConv} holds, for all $k, k' \ge 0$.
    Scheffé's lemma shows that this pointwise convergence can be
    reinforced to a convergence in $\ell_1(\NN \times \NN)$, which readily
    entails (c).

    Therefore, Proposition~6.2 in \cite{garavaglia20} shows that
    $\tilde{\mathcal{G}}^N$ converges in probability in the local weak
    sense towards a marked Galton--Watson tree $\widetilde{\mathcal{T}}$
    where each vertex $u$ has:
    \begin{enumerate}
        \item a $\mathrm{Poisson}(S_0R_0 + I_0 \bar{R}_0)$ number of
            offspring (with edges oriented from the children towards the
            parents); and
        \item an independent mark $\tilde{m}_u$ distributed as
            $\abs{\widehat{\mathcal{P}}}$ for the root and as the
            size-biasing of $\abs{\widehat{\mathcal{P}}}$ for other
            vertices.
    \end{enumerate}
    Note that in this tree there is no distinction between ($I$) and
    ($S$) vertices since part of the marking has been removed.

    \medskip
    \noindent
    \textbf{Step 2.} We now show that $\mathcal{G}^N$ (with the full
    marking) converges to a tree obtained by marking the limiting
    Galton--Watson tree as in Proposition~\ref{prop:configurationModel}.
    Let $x, y \in [N]$, and let $\tilde{\mathcal{G}}^N(x)$ and
    $\tilde{\mathcal{G}}^N(y)$ be the subgraphs of $\tilde{\mathcal{G}}^N$ 
    induced by all the vertices with an oriented path to $x$ and $y$ respectively.
    By construction, each vertex in $\mathcal{G}^N$ (and thus
    each vertex in $\tilde{\mathcal{G}}^N(x)$ and $\tilde{\mathcal{G}}^N(y)$) is indexed
    by an element of $[N]$ that we call its \emph{label}. Recall the
    notation $[G]_r$ for the ball of radius $r$ of a pogm $G$ around the
    pointed vertex. Let $B^N_r$ be the event ``the labels of the vertices
    in $[\tilde{\mathcal{G}}^N(x)]_r$ and $[\tilde{\mathcal{G}}^N(y)]_r$ are distinct''.
    On this event Proposition~\ref{prop:configurationModel} shows that,
    conditional on the unmarked graphs $[\tilde{\mathcal{G}}^N(x)]_r$ and
    $[\tilde{\mathcal{G}}^N(y)]_r$, the marks $(m_u)_u$ of the vertices
    in $[\mathcal{G}^N(x)]_r$ and $[\mathcal{G}^N(y)]_r$ are independent and
    \begin{equation} \label{eq:marking}
        m_u \sim (Z, \mathcal{P}, X) \text{ conditioned on $\abs{\widehat{\mathcal{P}}} = \tilde{m}_u$}.
    \end{equation}
    Furthermore the lengths of the edges going out of $u$ are sampled uniformly 
    among the atoms of $\widehat{\mathcal{P}}_u$. Now, the first part
    of the proof and Lemma~\ref{lem:convMeasure} show that 
    $(\tilde{\mathcal{G}}^N(x), \tilde{\mathcal{G}}^N(y))$ converges in
    distribution to two independent copies $(\widetilde{\mathcal{T}}_1,
    \widetilde{\mathcal{T}}_2)$ of the limit Galton--Watson tree. 
    Provided that $\P(B^N_r) \to 1$, this shows that in distribution
    \[
        ([\mathcal{G}^N(x)]_r, [\mathcal{G}^N(y)]_r) \to
        ([\mathcal{T}_1]_r, [\mathcal{T}_2]_r)
    \]
    where the tree $\mathcal{T}_i$ is obtained out of $\widetilde{\mathcal{T}}_i$ 
    by adding marks and edge lengths as in \eqref{eq:marking} and
    removing edges pointing to an ($I$) vertex. In turn,
    Lemma~\ref{lem:convMeasure} proves that $\mathcal{G}^N$ converges to
    $\mathcal{T}_1$ in probability in the local weak sense. 
    It remains to show that $\P(B^N_r) \to 1$. This
    result is actually shown as a step in the proof of Proposition~6.2 from
    \cite{garavaglia20} that we have used in our Step 1. More precisely,
    the proof of \cite[Lemma~6.4]{garavaglia20} shows that, with
    probability going to $1$, the balls of radius $r$ of two uniformly
    chosen vertices in the directed configuration model do not intersect,
    which is the result we need here. Let us explain heuristically why we
    expect this result to hold. The ball $[\mathcal{G}^N(x)]_r$ can be
    constructed by exploring the graph starting from $x$, following
    the in-edges in reverse direction, and pairing them with out-edges.
    Each time a new in-edge is explored, it is paired with an out-edge
    chosen uniformly from the unpaired out-edges in the graph. Since the
    total number of edges explored in $[\mathcal{G}^N(x)]_r$ and
    $[\mathcal{G}^N(y)]_r$ is negligible w.r.t.\ the total number of
    edges in $\mathcal{G}^N$ (and since no vertex in $\mathcal{G}^N$ has
    a number of out-edges of order $N$) the probability that the same vertex
    is explored both in $[\mathcal{G}^N(x)]_r$ and $[\mathcal{G}^N(y)]_r$
    vanishes as $N \to \infty$. This argument is made rigorous in the
    proof of \cite[Lemma~6.4]{garavaglia20}.

    \medskip
    \noindent
    \textbf{Step 3.}
    Let $\mathcal{T}$ be distributed as the local weak limit of
    $\mathcal{G}^N$ from the previous step. Our last task is to
    connect the distribution of $\mathcal{T}$ to that of the Poisson tree
    $\mathcal{H}$ from Section~\ref{def:poisson-tree}. Let us first take
    care of the root $\emptyset$. By definition of
    $\widetilde{\mathcal{T}}$, $\tilde{m}_\emptyset \sim
    \abs{\widehat{\mathcal{P}}}$ and conditional on
    $\tilde{m}_\emptyset$, $m_\emptyset \sim (Z, \mathcal{X})$
    conditioned on $\abs{\widehat{\mathcal{P}}} = \tilde{m}_\emptyset$.
    This readily shows that the mark of the root is distributed as $(Z,
    \mathcal{X})$, so that in particular it is of type ($I$) and ($S$)
    with probability $I_0$ and $S_0$ respectively.

    We now turn to some non-root vertex $u \in \widetilde{\mathcal{T}}$.
    Recall that its mark $\tilde{m}_u$ has the size-biased distribution
    of $\abs{\widehat{\mathcal{P}}}$ and that $m_u = (Z_u, \mathcal{X}_u)$ 
    is obtained as in \eqref{eq:marking}. Let $A_u$ be the length of its
    unique out-edge, which is uniformly chosen among the atoms of
    $\widehat{\mathcal{P}}_u$. We have
    \begin{align*}
        \EE\Big[ \phi\big( A_u, Z_u, \mathcal{X}_u \big) \Big] 
        &= 
        \EE\Big[ \EE\Big[
            \frac{1}{\abs{\widehat{\mathcal{P}}}}\int \phi\big( a, Z, \mathcal{X} \big) 
        \widehat{\mathcal{P}}(\diff a) \:\Big|\: \abs{\widehat{\mathcal{P}}} = \tilde{m}_u \Big] \Big] \\
        &= \frac{1}{\EE[\abs{\widehat{\mathcal{P}}}]} 
        \EE\Big[ \int \phi\big( a, Z, \mathcal{X} \big) 
        \widehat{\mathcal{P}}(\diff a) \Big] \\
        &= 
        \begin{multlined}[t]
        \frac{1}{S_0R_0 + I_0\bar{R}_0} 
        S_0 \EE\Big[ \int \phi\big( a, 0, \mathcal{X} \big) 
        \mathcal{P}(\diff a) \Big]  \\
        +
        \frac{1}{S_0R_0 + I_0\bar{R}_0} 
        I_0 \EE\Big[\int_0^\infty g(u) \int_u^\infty
            \phi\big( a-u, u, \mathcal{X} \big) \mathcal{P}(\diff a)\diff u \Big].
        \end{multlined} \\
        &= 
        \begin{multlined}[t]
        \frac{S_0R_0}{S_0R_0 + I_0\bar{R}_0} 
        \frac{1}{R_0} \EE\Big[ \int \phi\big( a, 0, \mathcal{X} \big) 
        \mathcal{P}(\diff a) \Big]  \\
        +
        \frac{I_0\bar{R}_0}{S_0R_0 + I_0\bar{R}_0} 
        \frac{1}{\bar{R}_0} \EE\Big[\int_0^\infty g(u) \int_u^\infty
            \phi\big( a-u, u, \mathcal{X} \big) \mathcal{P}(\diff a)\diff u \Big],
        \end{multlined}
    \end{align*}
    where in the first line we have used \eqref{eq:marking} and that
    $A_u$ is a uniform atom of $\widehat{\mathcal{P}}_u$, in the second
    line that $\tilde{m}_u$ has the size-biased distribution of
    $\widehat{\mathcal{P}}$ and in the third line the definition of
    $\widehat{\mathcal{P}}$ of \eqref{eq:shiftedPP}.
    The result now follows upon identifying the terms. The prefactor in each
    term of the sums corresponds to the probability that $Z_u = 0$ or $Z_u >
    0$, that is, that vertex $u$ is of type ($S$) or ($I$). Since
    the total number of offspring in $\widetilde{\mathcal{T}}$ follows a Poisson
    distribution with parameter $S_0R_0 + I_0\bar{R}_0$, the number
    of ($S$) and ($I$) offspring are independent Poisson random variables
    with means $S_0 R_0$ and $I_0 \bar{R}_0$ respectively. Moreover
    \[
        \frac{1}{R_0} \EE\Big[ \int \phi\big( a, \mathcal{X} \big) 
        \mathcal{P}(\diff a) \Big]
        =
        \EE\big[ \phi( W, \mathcal{X}^{\star} ) \big]
    \]
    where $(W, \mathcal{X}^{\star})$ has the Campbell measure of
    Definition~\ref{def:campbel}. Thus any ($S$) individual in $\mathcal{T}$
    has an edge length and mark distributed as $(W,
    \mathcal{X}^{\star})$ as in $\mathcal{H}$. Similarly, 
    \begin{align*}
        \frac{1}{\bar{R}_0} \EE\Big[\int_0^\infty g(u) \int_u^\infty
            &\phi( a-u, u, \mathcal{X} ) \mathcal{P}(\diff a)\diff u \Big] \\
        &= \frac{1}{\bar{R}_0} \int_0^\infty g(u) \int_u^\infty
        \tau(a) \EE\big[ \phi( a-u, u, \mathcal{X}^{(a)} ) \big] \diff a\diff u  \\
        &= \int_0^\infty \int_0^\infty G(v,u)
        \EE\big[ \phi( v, u, \mathcal{X}^{(u+v)} ) \big] \diff v\diff u,
    \end{align*}
    where $\mathcal{X}^{(a)}$ has the Palm distribution of $\mathcal{X}$
    at $a$, and $G$ is the probability density defined in \eqref{def:g}.
    In the second line we have used the definition of the Palm measure.
    Identifying the terms, the mark of an ($I$) individual is obtained as
    that defined for $\mathcal{H}$.
\end{proof}

\section{Convergence of the historical process}
\label{SS:proofMainThm}

We can now state and prove our main result. Let us introduce
the historical process as the following empirical measure
\begin{equation}\label{eq:def-historical}
    H^N \ \coloneqq \ \sum_{x\in[N]} \indic{\sigma^N_x<\infty} \ \delta_{{\cal A}^{N}_x}.
\end{equation}
We also define the historical process at time $t \ge 0$ as the historical
process of all individuals infected before time $t$,
\begin{equation*}
    H^N_t \ \coloneqq \ \sum_{x\in[N]} \indic{\sigma^N_x \le t} \ \delta_{{\cal A}^{N}_x}.
\end{equation*}

\begin{theorem}[Convergence of the historical process]\label{thm:hist-process}
Let ${\cal A}^\infty$ be the limiting ancestral process in the Poisson
tree $\mathcal{H}$ and let $(\sigma_0 , {\cal X})$ denote a pair of
independent random variables where $-\sigma_0$ is distributed according
to the density $g$. 

\begin{enumerate}
    \item[(i)]
For any $t \ge 0$ we have
\[
    \frac{1}{N} H_t^N  \ \longrightarrow \ S_0 \P( \sigma^\infty \le t) 
    {\cal L}\left({\cal A}^\infty \mid \sigma^\infty \le t \right) \ + \  I_0 {\cal L}( \sigma_0 , {\cal X} ).
\]
where ${\cal L}\left({\cal A}^\infty \mid \sigma^\infty \le t \right) $
is the law of the random variable ${\cal A}^\infty$ conditioned on the
event $\{ \sigma^\infty \le t \}$, and the convergence is in distribution
for the weak topology.

\item[(ii)]
If $(c(t);\, t \ge 0)$ converges as $t \to \infty$ we have that 
\[
    \frac{1}{N} H^N  \ \longrightarrow \ S_0 \P( \sigma^\infty <\infty) 
    {\cal L}\left({\cal A}^\infty \mid \sigma^\infty <\infty \right) \ + \  I_0 {\cal L}( \sigma_0 , {\cal X} ).
\]
in distribution
for the weak topology.
\end{enumerate}
\end{theorem}

The convergence result in (ii) is stronger than that in (i), but requires
the mild assumption that the contact rate converges. Point (i) of the
previous result is sufficient to derive the limit of the age structure of
the epidemic, our Theorem~\ref{thm:ageStructure}. However, it is not
sufficient to prove that the total number of individuals infected during
the epidemic converges. This is a very well-studied quantity in epidemic
modeling, referred to as the \emph{final size} of the epidemic 
\cite{ma2006generality, arino2007final}, and our motivation for deriving
point (ii) is the following corollary.

\begin{corollary}[Final size of the epidemic]
    Suppose that $(c(t);\, t \ge 0)$ converges as $t \to \infty$, then
    \[
        \frac{1}{N} \sum_{x \in [N]} \indic{\sigma^N_x < \infty}
        \ \longrightarrow\ 
        1 - \lim_{t \to \infty} S(t),
    \]
    in distribution as $N \to \infty$.
\end{corollary}

\begin{proof}
    By Theorem~\ref{thm:hist-process}, point (ii), we have that 
    \[
        \frac{1}{N} \sum_{x \in [N]}  \indic{\sigma^N_x < \infty}
        \ \longrightarrow\ 
        S_0 \P(\sigma^\infty < \infty) + I_0
        = \lim_{t \to \infty} S_0 \P(\sigma^\infty \le t) + I_0.
    \]
    By Proposition~\ref{prop:dual},
    \[
        S_0 \P(\sigma^\infty \le t) = B(t) = S_0 - S(t),
    \]
    so that 
    \[
        \lim_{t \to \infty} S_0 \P(\sigma^\infty \le t) + I_0 
        = I_0+S_0 - \lim_{t \to \infty} S(t) = 1 - \lim_{t \to \infty} S(t).
    \]
\end{proof}

To prove the convergence of the historical process, we see the ancestral
process $\mathcal{A}^N_x$ as a functional of the pogm graph
$\mathcal{G}^N(x)$ rooted at $x$. Provided we can show that the
mapping taking a pogm graph to its active geodesic enjoys some
appropriate continuity, the convergence of the historical process will
follow from the local weak convergence of the infection graph
$\mathcal{G}^N$. 

For a deterministic pogm graph $G$, we can define an infection process by
attaching to each edge $e$ a uniform infection intensity $s_e$ which
determines if the edge is active or not, as in Section~\ref{sec:def:active-goedesics}. 
It will be convenient to work conditional on $(s_e)$ and to think of
these infection intensities as a marking of the edges of the graph. It is
straightforward to extend the definitions and results from
Section~\ref{SS:localTopology} to include this marking, and that the
convergence of the infection graph in Proposition~\ref{prop:LWconvergence} 
also remains valid for this extended marking: the infection graph
$\mathcal{G}^N$, marked with uniform infection intensities, converges in
the local weak sense to the tree $\mathcal{T}$, also marked with uniform
infection intensities. 

For a pogm graph with fixed infection intensities, $G$, we can define 
$\mathcal{A}(G)$ as the ancestral process of $G$, which records the
infection times along the active geodesic leading to the pointed vertex,
as defined in Section~\ref{sect:ancestral-path}. We also define
$\sigma(G)$ as the length of the corresponding active geodesic. We can
now prove that the ancestral process is a continuous functional of the
local graph topology.

\begin{lemma} \label{lem:continuityAncestral}
    Let $f$ be a continuous bounded functional on the space of ancestral
    paths. Then for any $t > 0$ the map 
    \begin{equation} \label{eq:geodesicMap}
        G \mapsto f(\mathcal{A}(G)) \indic{\sigma(G) < t}
    \end{equation}
    is continuous at a.e.\ realization $G$ of the tree $\mathcal{T}$.
    If the function $(c(t);\, t \ge 0)$ converges as $t \to \infty$, then $f$ is
    continuous at a.e.\ realization of $\mathcal{A}(\mathcal{T})$.
\end{lemma}

\begin{proof}
The tree $\mathcal{T}$ is either made of a singled ($I$) vertex, or is a
copy of $\mathcal{H}$. Clearly, in the former case the result holds so
that it remains to show it for almost every realization of $\mathcal{H}$.
For some pogm tree $G$, if $d$ is the genealogical distance and $\pi_v$
denotes the unique path from $v$ to the root, let 
\begin{equation*} 
    M_r(G) \coloneqq 
    \min_{\substack{u \in G\\ d(\emptyset, u) = r}} \abs{\pi_v}
\end{equation*}
be the length of the shortest path from a vertex at distance $r$ to the
root. We start by showing in Step 1 that, almost surely, either
$\mathcal{H}$ is finite, or 
\begin{equation} \label{eq:minBRW}
    M_r(\mathcal{H}) \to \infty \quad\text{as $r \to \infty$}.
\end{equation}
Then, in Step 2, we show that \eqref{eq:geodesicMap} is continuous
for almost all graphs $G$ verifying this property. Under the additional
assumption that $(c(t);\, t \ge 0)$ converges, we prove that $f$ is
continuous in Step~3 and Step~4.

\medskip
\noindent
\textbf{Step 1.} Let $(V_r;\, r \ge 0)$ be the process that records the
ages of the $(S)$ vertices of the Poisson tree $\mathcal{H}$, defined as
\[
    V_r \coloneqq \sum_{\substack{u \in \mathcal{H},\, d(u,\emptyset) = r\\\text{$u$
    of type $(S)$}}} \delta_{\abs{\pi_u}},
\]
Then $(V_r)_{r \ge 0}$ is a branching random walk
with Poisson($R_0S_0$) offspring distribution, and it follows from general results 
that, conditional on non-extinction, its minimum drifts to $\infty$, see
for instance Theorem~5.12 in \cite{shi_2015}.
As $\mathcal{H}$ is obtained by attaching independently to any unmarked
vertex a $\mathrm{Poisson}(I_0 \bar{R}_0)$ distributed number of $(I)$
leaves, this shows that \eqref{eq:minBRW} also holds.
This completes Step 1.

\medskip
\noindent
\textbf{Step 2.} First let us note that the marks $(s_e)$, representing the infection intensities, are independent of the structure of the tree $\mathcal{T}$ and the lengths of its edges, so $\mathcal{T}$ almost surely satisfies the following property, for all $r\in \mathbb{N}$:
\begin{equation} \label{eq:sMarksNice}
	\{s_e, \, e\text{ edge of }\mathcal{T}\} \cap \{ c(t), \, t\in A_r^{\mathrm{pot}}\}=\emptyset,
\end{equation}
where $A_r^{\mathrm{pot}}$ is the (a.s.\ finite) set of lengths of all
paths from $(I)$ vertices at distance at most $r$ from the root to $(S)$
vertices. If a tree $G$ satisfies this property, then it is clear that
for any sequence $G^N\to G$, for $N$ large enough, the $r$-neighborhood
of the root in $G^N$ has the same structure as that of $G$ and in this
neighborhood, a path from an $(I)$ vertex to an $(S)$ vertex is open in
$G^N$ if and only if it is open in $G$.

Fix some tree $G$ satisfying \eqref{eq:sMarksNice}, which is either
finite or fulfills \eqref{eq:minBRW}, and a sequence $G^N$ converging to
$G$ in $\mathscr{H}$. We need to prove that
\begin{equation} \label{eq:goalContinuity}
    f(\mathcal{A}(G^N)) \indic{\sigma(G^N) < t} \to f(\mathcal{A}(G)) \indic{\sigma(G) < t}
\end{equation}
It is readily checked that \eqref{eq:goalContinuity} holds if $G$ is a
finite tree.
Suppose that $G$ is infinite. If $\sigma(G) < \infty$, let
$r$ be such that $M_r(G) > \sigma(G)$. In particular there is an active
path from an ($I$) vertex in $[G]_r$ to the root. The convergence
$[G^N]_r \to [G]_r$ entails that, for $N$ large enough, there is also an
active path from an ($I$) vertex in $[G^N]_r$ to its root whose length
converges to $\sigma(G)$, and that all other active paths from $G^N
\setminus [G^N]_r$ to the root have a length larger than $\sigma(G)$ and
thus cannot be the active geodesic. We are back to the case of a finite
tree where \eqref{eq:goalContinuity} is readily checked. Finally, if
$\sigma(G) = \infty$, fix $r$ such that $M_r(G) > t$. The convergence $[G^N]_r
\to [G]_r$ now entails that, for $N$ large enough, there is no active
path from an ($I$) vertex in $[G^N]_r$ to the root, so that $\sigma(G^N)
> t$. This shows that in all three cases \eqref{eq:goalContinuity} holds
and proves the first part of the result. We have also shown that $f$ is
continuous at every $G$ that fulfills \eqref{eq:minBRW} and has an active
geodesic, and that if $G^N \to G$ where $G$ fulfills \eqref{eq:minBRW}
and has no active geodesic, necessarily $\sigma(G^N) \to \infty$.

\medskip\noindent
We now prove the second part of the result and assume that $(c(t);\, t
\ge 0)$ converges to a limit $c_*$ as $t \to \infty$. For a pogm graph
$G$ with infection intensities on its edges, we denote by $G_s$ the pogm
graph obtained by removing from $G$ all edges $e$ with an infection
intensity $s_e > s$. We proceed again in two steps. In Step~3 we show
that if $\mathcal{H}_{c_*}$ is infinite, then $\mathcal{H}$ has a.s.\ a
geodesic. In Step~4 we consider a pogm graph $G$ such that $G_{c_*}$ is
finite, and prove that if $G^N \to G$ then $\sigma(G^N)
\indic{\sigma(G^N) < \infty}$ is bounded. Before moving to the proof of
these two claims, let us show that they are sufficient to prove our
result. If $G$ is such that \eqref{eq:minBRW} holds and has an active
geodesic, by Step~2 $f$ is continuous at $G$. Therefore, by Step~3, $f$
is continuous at a.e.\ realization $G$ of $\mathcal{H}$ such that
$G_{c_*}$ is infinite, or such that $G_{c_*}$ is finite but $G$ has an
active path to the root. It remains to consider the case where $G_{c_*}$
is finite and $G$ has no active path. If $G^N \to G$, by Step~4
$\sigma(G^N) \indic{\sigma(G^N)<\infty}$ is bounded and by Step~2
$\sigma(G^N) \to \infty$. Necessarily, $\sigma(G^N) = \infty$ and
$f(\mathcal{A}(G^N)) = f(\emptyset) = f(\mathcal{A}(G))$ for $N$ large
enough, showing that $f$ is continuous at $G$. Our only remaining task is
now to show the previous two claims.

\medskip\noindent
\textbf{Step 3.} We show that a.e.\ realization of $\mathcal{H}$ such
that $\mathcal{H}_{c_*}$ is infinite has an active geodesic. There are
two trivial cases that we easily exclude: if $c_* = 0$,
$\mathcal{H}_{c_*}$ cannot be infinite, and if $\sigma^\infty$ has
bounded support, our result is trivial because the epidemic stops a.s.\
after a finite time. Now, for any $s$, by standard properties of Poisson
random variables, the graph $\mathcal{H}_s$ is again a Galton--Watson
tree with Poisson distributed offspring and the graph $\mathcal{H}$ is
obtained by grafting independently on each ($S$) vertex of
$\mathcal{H}_s$:
\begin{itemize}
    \item a $\mathrm{Poisson}( S_0 R_0 (1-s) )$ distributed number of
        copies of $\mathcal{H}$; and
    \item a $\mathrm{Poisson}( I_0 \bar{R}_0 (1-s) )$ distributed number
        of ($I$) vertices.
\end{itemize}
Furthermore, each of these trees is connected to $\mathcal{H}_s$ through a unique
edge whose infection intensity is uniform on the interval $(s, 1)$. 
Note that when $s$ increases, so does the number of edges in $\mathcal{H}_s$, therefore we have $\{ \text{$\mathcal{H}_{c_*-\epsilon}$ is infinite} \}\subset \{ \text{$\mathcal{H}_{c_*}$ is infinite} \}$
for each $\epsilon>0$.
Furthermore, by studying the extinction probability of these Galton--Watson trees, we readily see that the probability $\P(\text{$\mathcal{H}_{s}$ is infinite})$ is a continuous function of $s$, which implies that
\[
	\P\Big(\{ \text{$\mathcal{H}_{c_*}$ is infinite} \}\setminus \bigcup_{\epsilon > 0} \{ \text{$\mathcal{H}_{c_*-\epsilon}$ is infinite} \} \Big) = 0.
\]
In other words, up to a null probability event, we have
\[
    \bigcup_{\epsilon > 0} \{ \text{$\mathcal{H}_{c_*-\epsilon}$ is
    infinite} \}
    = \{ \text{$\mathcal{H}_{c_*}$ is infinite} \},
\]
therefore without loss of generality, we can can consider a realization of $\mathcal{H}$ and an $\epsilon>0$ such that $\mathcal{H}_{c_*-\epsilon}$ is infinite.
Now let $T$ be such that $\abs{c(t) - c_*} < \epsilon / 2$
for $t \ge T$. On the event that $\mathcal{H}_{c_*-\epsilon}$ is
infinite, a.s.\ we can find a subtree $G$ of $\mathcal{H}$ grafted on
$\mathcal{H}_{c_*-\epsilon}$ such that $\sigma(G) > T$ and such that the
edge connecting $G$ to $\mathcal{H}_{c_*-\epsilon}$ has an infection
intensity in $(c_*-\epsilon, c_*-\epsilon / 2)$. (We have used that $\sigma^\infty$
has unbounded support.)
Let us write $e_1$ for this edge and let $\pi = (e_1,\dots,e_n)$ be the unique path in
$\mathcal{H}$ leading to the root and extending the active geodesic in
$G$. Since $\sigma(G) > T$, for each edge $e_k$ of $\pi$, we have $\abs{\tau_k(e_k)} > T$, so that
$c(\abs{\tau_k(e_k)}) > c_* - \epsilon / 2$ and the edge is open.
Therefore there exists a.s.\ an active path in $\mathcal{H}$ if
$\mathcal{H}_{c_*-\epsilon}$ is infinite.

\medskip
\noindent
\textbf{Step 4.} We now consider a pogm tree $G$ such that $G_{c_*}$ is
finite and $G$ has no active path. If $G^N \to G$, we need to show that 
for $N$ large enough $\sigma(G^N) = \infty$, that is, that $G^N$ has no
active path. Since $G_{c_*}$ is finite, there exists $r$ such that $[G_{c_*}]_r =
G_{c_*}$. By definition of $G_{c_*}$, all edges $e$ in $G$ pointing to a leaf
vertex of $G_{c_*}$ have an infection intensity
$s_e > c_*$. Since $[G^N]_{r+1} \to [G]_{r+1}$, the same holds true for $G^N$
for $N$ large enough. If $\pi^N$ is an active path leading from an $(I)$ vertex to the root in
$G^N$, by Step~2, it has to satisfy $\abs{\pi^N} \to \infty$. However,
any path $\pi^N = (e^N_1, \dots, e^N_n)$ ending at the root in $G^N$ with
$\abs{\pi^N} \to \infty$ includes an edge $e^N_k$ pointing to a leaf vertex
of $G_{c_*}$. On one hand, since $\abs{\pi^N}\to \infty$ we have
$c(\abs{\tau_k \pi^N}) \to c_*$. On the other hand,
since $[G^N]_{r+1} \to [G]_{r+1}$, we have $\liminf_{N \to \infty}
s_{e_k^N} > c_*$, so that this path is not active for large enough $N$,
proving that there exists no such active paths.
\end{proof}

\begin{proof}[Proof of Theorem~\ref{thm:hist-process}]
    For point (i), we have that for any bounded continuous functional
    $f$,
    \[
        \frac{1}{N} \sum_{x=1}^N f(\mathcal{A}^N_x)
        \indic{\sigma^N_x < t}
        \ \longrightarrow\
        \EE[ f(\mathcal{A}(\mathcal{T})) \indic{\sigma(\mathcal{T}) < t} ]
    \]
    by using Lemma~\ref{lem:double-continuous-mapping} with the mapping
    $\Phi \colon G \mapsto f(\mathcal{A}(G)) \indic{\sigma(G) < t}$ (which is
    a.e.\ continuous by Lemma~\ref{lem:continuityAncestral}) and the sequence
    $(\mathcal{G}^N)_N$ (which converges in the local weak sense by
    Proposition~\ref{prop:LWconvergence}). This in turn proves 
    that
    \[
        \frac{1}{N} \sum_{x \in [N]} \indic{\sigma^N_x < t} \delta_{\mathcal{A}^N_x} \to  
        \P(\sigma(\mathcal{T}) < t) \mathcal{L}(\mathcal{A}(\mathcal{T}) \mid \sigma(\mathcal{T}) < t)
    \]
    in probability for the weak topology (see for instance \cite[Theorem 4.11]{Kal17}).
    The proof of the first point is ended by noting that, with
    probability $I_0$ we have $\mathcal{A}(\mathcal{T}) = (-Z,
    \mathcal{X})$, whereas with probability $S_0$ we have $\mathcal{T} =
    \mathcal{H}$.

    The proof of point (ii) is the same as for point (i), but replacing 
    the map $G \mapsto f(\mathcal{A}(G)) \indic{\sigma(G) < t}$ by the
    map $G \mapsto f(\mathcal{A}(G))$ and using the second part of
    Lemma~\ref{lem:continuityAncestral} for the continuity.
\end{proof}

We can now prove \Cref{thm:ageStructure} using \Cref{thm:hist-process}.
Recall the notation
\[
    \mu_t^N = \sum_{x \in [N]} \indic{\sigma^N_x \le t} \delta_{(t-\sigma^N_x, X_x(t-\sigma^N_x))}
\]
for the empirical distribution of ages and compartments at time $t$, and the
notation
\begin{equation}\label{eq:ymu}
    Y_t^N(i) = \sum_{x \in [N]} \indic{\sigma^N_x \le t, X_x(t-\sigma^N_x) = i} 
    = \mu_t^N\big([0, \infty), \{i\}\big)
\end{equation}
for the number of individuals in compartment $i$ at time $t$. Note that
$\mu_t^N=\mu^N_t(\d a,\d i)$ can be written in terms of $H^N$ as follows
\begin{equation} \label{eq:mut-from-mu}
    \mu^N_t = N\int \indic{0\leq \sigma_0\leq t} 
    \delta_{(t-\sigma_0,X(t-\sigma_0))}\,H^N(\d \pi),
\end{equation}
where $\pi=(\sigma_\ell,\mathcal{X}_\ell)_{\ell=0}^k=(\sigma_\ell,(\mathcal{P}_\ell, X_\ell))_{\ell=0}^k$ denotes a generic ancestral path.

\begin{proof}[Proof of \Cref{thm:ageStructure}]
  By \Cref{thm:hist-process} and \eqref{eq:mut-from-mu}, we get for fixed $t$ and $i\in \mathcal{S}$,
  \[
    \mu^N_t(\d a,\{i\}) \longrightarrow S_0\P(t-\sigma^\infty \in \d a) \P(X(a) = i),
  \]
  where $\sigma^\infty$ is the length of the active geodesic in
  $\mathcal{H}$, and $X$ is a life-cycle process. Using \Cref{prop:dual},
  we can further identify $S_0\P(t-\sigma^\infty \in \d a)=n(t,a)\,\d a$,
  proving finite-dimensional convergence of $(\mu^N_t/N;\, t \ge 0)$.
  Because of the expressions of $Y^N_t(i)$ in terms of $\mu^N_t$ in
  \eqref{eq:ymu}, identification of their limit is trivial. All there is
  to check is tightness of the processes. 

  The tightness for $(\mu^N_t / N;\, t \ge 0)$ will follow from
  that of $(Y^N_t(i)/N;\, t \ge 0)$. Recall that the compartments of the
  life-cycle process enjoy an ``acyclic orientation'' property.  See
  statement before Theorem~\ref{thm:ageStructure}. Writing $i \preceq j$
  if $j$ can be accessed from $i$, the process
  \begin{equation} \label{eq:rearrangedCompartments}
      \sum_{j: \, i \preceq j} \frac{1}{N} Y_t^N(j),
  \end{equation}
  is nondecreasing in time. Since the finite-dimensional marginals of
  this nondecreasing process converge to the expression on the RHS of
  \eqref{eq:limitCompartments}, tightness follows provided we can show
  that this limit is continuous, see for instance Theorem~3.37,
  Chapter~VI of \cite{JS03}. For the continuity, for $t \ge 0$ and $h > 0$, 
  using H\"older inequality
  \begin{multline*}
      \abs*{\int_0^\infty n(t,a) p(a,i) \diff a - \int_0^\infty n(t+h,a) p(a,i)
      \diff a} \\
      \begin{aligned}
      &\le \int_0^\infty \abs{n(t,a) - n(t+h,a)} \diff a  \\
      &\le \int_0^h n(t,a) + n(t+h,a) \diff a  
      + \int_0^\infty \abs{n(t,a) - n(t, a + h)} \diff a.
      \end{aligned}
  \end{multline*}
  The first term can be made small using for instance
  \eqref{eq:absoluteContinuity}, whereas for the second term one can use
  that translation operators are continuous on $L^1$. We proceed in a
  similar way for $h < 0$. The tightness of $Y^N_t(i)/N$ follows by
  subtracting the processes in \eqref{eq:rearrangedCompartments} in an
  appropriate way. 

  Let us turn to the tightness of $(\mu^N_t / N;\, t \ge 0)$. We will use
  a tightness criterion for measure-valued processes in
  \cite{vaillancourt90}. This criterion is stated for measures on a
  compact space, but can be easily adapted to our setting by considering
  a compactification of $\RR_+ \times \mathcal{S}$ and noting that the
  limit of our sequence \eqref{eq:limitCompartments} has no mass at
  infinity. According to Lemma~3.2 in \cite{vaillancourt90}, it is
  sufficient to check tightness of the processes $(Y^N_t(\phi,i)/N;\, t
  \ge 0)$, for $\phi \colon \RR_+ \to \RR$ uniformly continuous and where
  \[
      \forall t \ge 0,\quad Y^N_t(\phi, i) = \int_0^\infty \phi(a)
      \mu^N_t(\diff a, \{i\}).
  \]
  For $s < t$, we have
  \begin{align*}
      |Y^N_s(\phi,i) &- Y^N_t(\phi, i)| \\
      &= 
      \abs*{\sum_{x \in [N]} \phi(s-\sigma_x) \indic{X_x(s-\sigma_x)=i, \sigma_x \ge s}
                           - \phi(t-\sigma_x) \indic{X_x(t-\sigma_x)=i, \sigma_x \ge t}} \\
      &\le \sum_{x \in [N]} \abs{\phi(s-\sigma_x)-\phi(t-\sigma_x)}
      + \abs{Y^N_s(i) - Y^N_t(i)}.
  \end{align*}
  Tightness follows from the uniform continuity of $\phi$ and from the
  tightness of the sequence $(Y^N_t(i)/N;\, t \ge 0)$.
  \end{proof}

\bibliographystyle{plain}
\bibliography{refs}

\end{document}